\documentclass[11pt]{amsart}

\usepackage{amsthm,amssymb,amstext,amscd,amsfonts,amsbsy,amsxtra,latexsym,amsmath,mathtools}
\usepackage{fullpage}
\usepackage[english]{babel}
\usepackage[latin1]{inputenc}
\usepackage{verbatim}
\usepackage{graphicx}
\usepackage[all]{xy}
\numberwithin{figure}{section}

\newcommand{\field}[1]{\mathbb{#1}}
\newcommand{\N}{\field{N}}
\newcommand{\Z}{\field{Z}}
\newcommand{\R}{\field{R}}
\newcommand{\C}{\field{C}}
\newcommand{\Q}{\field{Q}}

\newcommand{\SL}{\operatorname{SL}}
\newcommand{\Sl}{\operatorname{s\ell}}

\renewcommand{\H}{\mathbb{H}}
\newcommand{\partialtheta}{\Theta}

\newcommand{\e}{\mathrm{e}}
\newcommand{\qq}{\mathfrak{q}}
\newcommand{\zz}{\mathfrak{z}}

\numberwithin{equation}{section}
\newtheorem{theorem}{\textbf{Theorem}}
\numberwithin{theorem}{section}
\newtheorem{lemma}[theorem]{\textbf{Lemma}}
\newtheorem*{defn}{Definition}

\newtheorem{proposition}[theorem]{\textbf{Proposition}}
\newtheorem{corollary}[theorem]{\textbf{Corollary}}
\newtheorem{example}[theorem]{\textbf{Example}}
\newtheorem*{remark}{Remark}
\newtheorem*{remarks}{Remarks}

\renewenvironment{proof}[1][Proof]{\begin{trivlist}
\item[\hskip \labelsep {\bfseries #1:}]}{\qed\end{trivlist}}

\begin{document}
\title{Negative index Jacobi forms and quantum modular forms}
\author{Kathrin Bringmann, Thomas Creutzig, and Larry Rolen}
\address{Mathematical Institute\\University of
Cologne\\ Weyertal 86-90 \\ 50931 Cologne \\Germany}
\email{kbringma@math.uni-koeln.de} 
\email{lrolen@math.uni-koeln.de}

\address{Department of Mathematical and Statistical Sciences\\ University of Alberta\\ 632 CAB\\ Edmonton, Alberta  T6G 2G1, Canada}
\email{creutzig@ualberta.ca}
\date{\today}
\thanks{The research of the first author was supported by the Alfried Krupp Prize for Young University Teachers of the Krupp foundation and the research leading to these results has received funding from the European Research Council under the European Union's Seventh Framework Programme (FP/2007-2013) / ERC Grant agreement n. 335220 - AQSER. The third author thanks the University of Cologne and the DFG for their generous support via the University of Cologne postdoc grant DFG Grant D-72133-G-403-151001011, funded under the Institutional Strategy of the University of Cologne within the German Excellence Initiative.}
\maketitle

\begin{abstract}
In this paper, we consider the Fourier coefficients of a special class of meromorphic Jaocbi forms of negative index considered by Kac and Wakimoto. Much recent work has been done on such coefficients in the case of Jacobi forms of positive index, but almost nothing is known for Jacobi forms of negative index. In this paper we show, from two different perspectives, that their Fourier coefficients have a simple decomposition in terms of partial theta functions. The first perspective uses the language of Lie super algebras, and the second applies the theory of elliptic functions. In particular, we find a new infinite family of rank-crank type PDEs generalizing the famous example of Atkin and Garvan. We then describe the modularity properties of these coefficients, showing that they are ``mixed partial theta functions'', along the way determining a new class of quantum modular partial theta functions which is of independent interest. In particular, we settle the final cases of a question of Kac concerning modularity properties of Fourier coefficients of certain Jacobi forms.  

\end{abstract}

\section{Introduction and statement of results }

Since the introduction of the theory of Jacobi forms by Eichler and Zagier \cite{EichlerZagier}, connections between Jacobi forms and modular-type objects have been a question of central interest, with applications to many areas including Siegel modular forms, mock modular forms, and Lie theory. In this paper we study the Fourier coefficients of a special family of negative index Jacobi forms. In particular, consider for $N\in\N,$ $M\in\N_0$ the functions 
\[\phi_{M,N}(z;\tau):=\frac{\vartheta\left(z+\frac12;\tau\right)^M}{\vartheta(z;\tau)^N},\]
where $\vartheta(z;\tau)$ is the usual Jacobi theta function \begin{equation}\label{JacobiThetaDefn}\vartheta(z;\tau):=-i\zeta^{-\frac12}q^{\frac18}(q)_{\infty}(\zeta )_{\infty}\left(\zeta^{-1}q\right)_{\infty}.\end{equation} Here, $q:=e^{2\pi i \tau} (\tau\in\H)$, $\zeta:=e^{2\pi i z} (z\in\C)$, and for $n\in\N_0\cup\{\infty\}$, $(a)_{n}:=(a;q)_{n}:=\prod_{j=0}^{n-1}\left(1-aq^j\right)$ is the $q$-Pochhammer symbol. The relationship between Jacobi forms and modular forms has appeared in many guises and stems back to important work on holomorphic Jacobi forms, which states that they have theta decompositions relating them to half-integral weight modular forms \cite{EichlerZagier}.  The situation for meromorphic positive index Jacobi forms is also well-understood; a meromorphic Jacobi form of positive index has Fourier coefficients which are \emph{almost mock modular forms}, which in turn are holomorphic parts of \emph{almost harmonic Maass forms} \cite{BringmannFolsomKacWaki, DMZ,Folsom, Ol,ZwegersThesis}. 
Loosely speaking, almost harmonic weak Maass forms are sums of harmonic weak Maass functions under iterates of the raising operator multiplied by almost holomorphic modular forms. 
In this paper, we describe new decompositions of the Jacobi forms $\phi_{M,N}(z;\tau)$ which complement this long history of previous work on positive index Jacobi forms in the much more mysterious case of negative index. In addition to being of interest in the subject of general Jacobi forms, here we give further applications of such decompositions, focusing on the special subfamily $\phi_N(\tau):=\phi_{0,N}(\tau)$ as they are of great interest in various areas such as number theory, representation theory, combinatorics, and physics. Here we outline just a few such occurrences. 

Firstly, for various choices of $N$, the functions $\phi_N$ are of combinatorial interest. In particular, the function $\phi_1$ is related to the famous Andrews-Dyson-Garvan crank generating function (see (\ref{crankproduct}) and (\ref{PhiCrank})), which was used by Andrews and Garvan to provide a combinatorial explanation for the Ramanujan congruences for the partition function \cite{AndrewsGarvanCrank}, as postulated by Dyson \cite{Dyson}. In this paper we describe relations between powers of the crank generating function with certain Appell-Lerch series, giving a new family of PDEs generalizing the ``rank-crank PDE'' of Atkin and Garvan \cite{AG} (see Theorem 1.3). This beautiful identity of Atkin and Garvan gives a surprising connection between the rank and crank generating functions which can be used to show various congruences relating ranks and cranks, as well as useful relations between the rank and crank moments \cite{AG}. We also note the other examples of similar PDEs related to combinatorics have shown up in, for example Section 3.2 of \cite{BringmannZwegers}, where the function $\phi_{1,3}(z;\tau)$ is studied in relation to overpartitions. 

Secondly, the functions $\phi_N$ contain information about certain affine vertex algebras and their associated affine Lie algebras  
 studied by Kac and Wakimoto \cite{KacWakimoto}.
More precisely, let $\mathcal S(1)$ be the graded vector space 
\[
 \mathcal S(1):= \text{Sym} \bigoplus_{n=0}^\infty \left(q^{n+\frac{1}{2}}\left(\zeta \C \oplus \zeta^{-1} \C\right)\right), 
\]
and $\mathcal S(N)= \mathcal S(1)^{\otimes N}$. The vector space $\mathcal S(N)$ can be given the structure of a vertex algebra, the bosonic $\beta\gamma$-ghost
vertex algebra of rank $N$ and central charge $-N$. 
The graded character of this vertex algebra has a nice product form in the domain $|q|^{\frac{1}{2}}<|\zeta_n|<|q|^{-\frac{1}{2}}$, where $\zeta_n:=e^{2\pi i z_n}$,

\begin{equation}\label{eq:charbg}
 \text{ch}[\mathcal S(N)](z_1,\dots,z_N;\tau) = \prod_{n=1}^N  q^{\frac{1}{24}}\text{ch}[\mathcal S(1)](z_n;\tau)=
 q^{\frac{N}{24}} \prod_{n=1}^N\frac{1}{\left(\zeta_n^{-1}q^{\frac{1}{2}}\right)_{\infty}\left(\zeta_nq^{\frac{1}{2}}\right)_{\infty}}.
\end{equation}
It specializes to $\phi_N(z;\tau)$ for the following choices:
\[
i^N\zeta^{\frac{N}{2}}q^{-\frac{N}{6}}\left(q\right)_\infty^{-N} \text{ch}[\mathcal S(N)]\left(z-\frac{\tau}{2},\dots,z-\frac{\tau}{2};\tau\right) = \phi_N(z;\tau).
\]
The algebra $\mathcal S(N)$ contains as commuting subalgebras the rank one Heisenberg vertex algebra $\mathcal H(1)$
and the simple affine vertex algebra of $\Sl(N)$ at level minus one, $L_{-1}(\Sl(N))$.
Note that minus one is not an admissible level in the case of $N=2$ and that these algebras do not form a mutually commuting pair inside
$\mathcal S(2)$. However for $N>2$ it was shown in \cite{Adamovic} that $L_{-1}(\Sl(N))$ and $\mathcal H(1)$ form such a mutually commuting pair inside $\mathcal S(N)$.
The character of the highest-weight module $\mathcal F_{\mu}$, $\mu\in\R$, of $\mathcal H(1)$ takes the form 
\begin{equation}\label{eq:charh}
\text{ch}\left[\mathcal F_{\mu}(z;\tau)\right] = \frac{\zeta^{\sqrt{N}\mu} q^{-\frac{\mu^2}{2}}}{q^{\frac{1}{24}}(q)_\infty},
\end{equation}
so that the Fourier coefficients in $\zeta$ of $\phi_N(z;\tau)$ immediately allow one to compute the multiplicity with which the 
character of $\mathcal F_{r/\sqrt{N}}$ appears. 
In physics language, such a multiplicity is called the \emph{branching function} of the coset $\mathcal S(N)/\mathcal H(1)$. 

This leads to the second conformal field theory and vertex algebra importance of decomposing a meromorphic Jacobi form.
One of the most interesting classes of vertex algebras is given by $V_k(\mathfrak g)$, the universal affine vertex algebra of the simple Lie
algebra $\mathfrak g$ at level $k\in\C$. For certain rational admissible levels, $V_k(\mathfrak g)$ is not simple and one instead 
prefers to study its simple quotient $L_k(\mathfrak g)$. The characters of irreducible highest-weight modules
at admissible level $L_k(\mathfrak g)$ are the sum expansions in special domains of meromorphic Jacobi forms \cite{KW2}.
Understanding these sum expansions is crucial in studying the modular data of the corresponding conformal field theory
\cite{CR1,CR2}. 

Fourthly, the functions $\phi_N$ appear in the denominator identities of affine Lie super algebras \cite{KW1}. 
In \cite{CR1,CR2} the denominator identity of $\widehat\Sl(2|1)$ was an essential ingredient to study the relations between characters of admissible level $L_k(\Sl(2))$, while we
use the identities for the family $\widehat{\Sl}(N|1)$ to prove one of our central theorems. 

Finally, the functions $\phi_N$ also occur in string theory; we only expound upon one example.
The reciprocal of the Igusa cusp form $\Phi_{10}(Z)$ ($Z\in\H_2$, the Siegel upper half plane of genus $2$) arises as the partition function of quarter-BPS dyons in the type II compactification on the product of a $K_3$ surface and an elliptic curve. Write
\begin{equation}\label{gusadec}
\frac1{\Phi_{10}(Z)}=\sum_{m=-1}^\infty \psi_m(z; \tau)\rho^m\qquad \left(Z=\left(\begin{matrix} \tau&z \\ z&w\end{matrix}\right),\ \rho=e^{2\pi iw}\right).
\end{equation}

For $m>0$, the Fourier coefficients of
the functions $\psi_m$ are the degeneracies of single-centered black holes and two-centered black holes with total magnetic charge invariant equal to $m$. This case is studied in pathbreaking work of Dabholkar, Murthy, and Zagier \cite{DMZ} (see also \cite{ManschotMoore} for the appearance of mock modular forms in the
context of quantum gravity partition functions and AdS3/CFT2, as well
as \cite{Manschot} for a relation between multi-centered black holes and mock
Siegel-Narain theta functions). The coefficient of $m=-1$ equals \[
\frac{-1}{\eta^{18}(\tau)\vartheta(z; \tau)^2}
\] 
(note that our theta function $\vartheta(z;\tau)$ differs from the theta function $\theta_1(z;\tau$) in the notation of \cite{DMZ} by a factor of $i$). Analogously to the case of Jacobi forms of positive index, one may view Theorem \ref {RankCrankEven} below as a decomposition giving a ``polar part'' but no ``finite part'' as described in \cite{BringmannFolsomKacWaki, DMZ} and stated in more detail in (\ref{KacWakiPositiveIndex}).
This is consistent with a string theoretic interpretation of $\psi_{-1}$ in (\ref{gusadec}) in that there are no single-centered black holes and the degeneracies are all interpreted as accounting for two-centered black holes (see \cite{DGN,Sen}).
In contrast, in the case $m>0$, the mock part of $\psi_m$ corresponds to single-centered black holes and the Appel Lerch sum corresponds to two-centered black hole bound states \cite{DMZ}.

Returning to the problem of studying $\phi_{M,N}$, we define its Fourier coefficients by
\begin{equation*}
\phi_{M,N}(z; \tau)=:\sum_{r \in\frac{M-N}2+\Z} \chi(M,N, r;\tau)\zeta^r
\end{equation*} 
and in particular we set $\chi(N,r;\tau):=\chi(0,N,r;\tau)$.  Note that wallcrossing occurs; the coefficients $\chi(M,N, r;\tau)$ are only well-defined if we fix a range for $z$. We show that the Fourier coefficients $\chi(M,N, r;\tau)$ can be described using \emph{partial theta functions} (i.e., sums over half a lattice which when summed over a full lattices becomes a theta function), whose modularity properties near the real line we also describe using quantum modular forms. Quantum modular forms were recently defined by Zagier in \cite{ZagierQuantum} (see also \cite{BrysonPitman,ChernRhoades,FOR,LiNgoRhoades}). Although the definition is not rigorous, Zagier gave a number of motivating examples. Roughly speaking, a \emph{weight $k$ quantum modular form} is a function $f\colon \mathcal Q\rightarrow\C$ for some subset $\mathcal Q\subseteq\mathbb P_1(\Q)$ such that for any $\gamma$ in a congruence subgroup $\Gamma$, the cocycle $f|_k(1-\gamma)$ extends to an open set of $\R$ and is ``nice'' (e.g. continuously differentiable, smooth, etc). In fact, our study of the modularity of the partial theta functions shows that they are what Zagier refers to as strong quantum modular forms, namely that they have a near-modular property for asymptotic expansions defined at every point in a subset of $\mathbb P_1(\Q)$. Moreover, this behavior comes from the ``leaking'' of modularity properties of a non-holomorphic Eichler integral defined on the lower half plane (see (\ref{EichlerIntDefn})).

 %As a striking example of a quantum modular form, consider Kontsevich's ``strange function''
%\[F(q):=\sum_{n\geq0}(q)_{n}.\]
%This function does not converge on any open subset of $\C$ but is a finite sum for $q$ any root of unity. Zagier showed that the values of $F(q)$ are essentially equal to the radial limits from inside the unit disk of the ``half-derivative'' of the Dedekind eta function \cite{ZagierVassiliev}. He used this to prove that the cocycle corresponding to $e^{\frac{\pi i x}{12}}F(e^{2\pi i x})$, initially defined on $\Q$, extends to an almost everywhere real-analytic function on $\R$.

Returning to the Fourier coefficients of $\phi_{M,N}$, we define a \emph{mixed partial theta function} to be a linear combination of quasi-modular forms multiplied with partial theta functions. These functions have known connections to many interesting combinatorial functions, such as concave and convex compositions \cite{AndrewsConcaveConvex}, unimodal sequences \cite{KimLovejoy,Stanley}, and stacks \cite{Wright}. Throughout, we abuse notation to say that any function is a modular form, partial theta function, mixed partial theta function, etc. if it is equal to such a function up to multiplication by a rational power of $q$. Our main result is the following.

\begin{theorem}\label{mainthm}
For any $N\in\N$, $M\in2\N_0$ and $r\in\Z$ with $M<N$, the functions $\chi(M,N, r;\tau)$ are mixed partial theta functions. 
\end{theorem} 

\begin{remarks}  
1. The quasi-modular forms appearing in the decomposition of the mixed partial theta functions are canonically determined by the Laurent expansion of the Jacobi form $\phi_{M,N}$ (see Theorem \ref{EvenDecomp}).

2. Using the techniques of this paper it is easy to relax the condition on $M$ to allow any natural number less than $N$, however we restrict to even $M$ for notational convenience. Together with \ref{mainthm}, and the works of \cite{BringmannFolsomKacWaki,DMZ,Ol}, this settles the final cases of modularity of Kac-Wakimoto characters raised in \cite{KacWakimoto}. 
\end{remarks}

We first consider the case of $\phi_N(z;\tau)$, which we study from two perspectives. Our first viewpoint describes the Fourier coefficients as derivatives of partial theta functions of a rescaled version
of the root lattice 
\[ 
A_{N-1} := \Z \alpha_1 \oplus \dots\oplus \Z\alpha_{N-1} 
\]
of $\Sl(N)$. Here, the $\alpha_n$ are the simple roots of $\Sl(N)$, which are
linear functionals on the Cartan subalgebra $\mathfrak h\cong \C^{N-1}$ of $\Sl(N)$.
The Gram matrix of $A_{N-1}$ is the Cartan matrix of $\Sl(N)$. We denote the bilinear form by $(\ \ | \ \ )$
and abbreviate $t^2:=(t|t)$ for $t$ in $A_{N-1}$.
For $r$ in $\Z$, we  define the subset  of $\frac{1}{N}A_{N-1}$ 
\begin{flalign}\label{eq:setTr}
 &&T_r := \left\{ \sum_{n=1}^{N-1}a_n\alpha_n  \Big | 
 a_n = \frac{(N-n)n}{2}-\frac{rn}{N}+(N-1)m_n, m_n \in\Z, \left(m_{N-1}-\frac{1}{2}\right)\left(r-\frac{1}{2}\right)\geq 0 \right\}
\end{flalign}
and its \emph{partial theta function}
\[ 
 P_r(\tau):=\sum_{t\in T_r } e^{t}q^{\frac{t^2}{2(N-1)}}. 
\]
The $e^t$ are functions on the Cartan subalgebra $\mathfrak h$ defined by $e^t\, : \, u\mapsto e^{t(u)}$
for $u$ in $\mathfrak h$. 
Its evaluation for $u\in\mathfrak h$
is then denoted by $P_r(u;\tau)$. 
We call $P_r$ a partial theta function, because the theta function obtained by summing over the complete
lattice $\frac{1}{N}A_{N-1}$, 
\[
  \theta_N(\tau):=\sum_{t\in \frac{1}{N}A_{N-1} } e^{t(0)}q^{\frac{t^2}{2(N-1)}},
 \]
is a modular form of weight $(N-1)/2$ for $\Gamma(M)$ with $M=N^2(N-1)/2$. This statement is true, since $\theta_N$ is the theta function of the lattice $\frac{1}{\sqrt{2M}}A_{N-1}$. The level of this lattice is $M$, and the modularity of theta functions of lattices is discussed for example in \cite{Ebeling}. 

Further let $\partial$ be the differential operator 
\begin{equation}\label{eq:partial}
\partial_\alpha e^\lambda := (\lambda|\alpha) e^\lambda \quad\text{and} \quad  \partial := \prod_{\alpha\in{\Delta_0^+}} \partial_\alpha, 
\end{equation}
where $\Delta_0^+$ is the set of positive roots of $\Sl(N)$. Finally, set $d_N:=\prod_{j=1}^Nj!$
and let $\text{sign}(r)=1$ if $r\geq 0$ and $-1$ otherwise. 
Then we have the following. 

\begin{theorem}\label{FourierCoeffsThomas}
For $N\geq 2$, the $r$-th Fourier coefficient 
of $\phi_{N}(z;\tau)$ is given by
\begin{equation}\nonumber
\begin{split}
\chi(N,r;\tau) &= \frac{i^N\mathrm{sign}\left(r-\frac{N}{2}\right) q^{\frac{-r^2}{2N}}}{\eta(\tau)^{N(N+1)}d_{N-1}} 
\sum_{t\in T_{r-\frac{N}{2}}} \prod_{\alpha\in{\Delta_0^+}} (t|\alpha)q^{\frac{t^2}{2(N-1)}}\\
&= \frac{i^N\mathrm{sign}\left(r-\frac{N}{2}\right) q^{\frac{-r^2}{2N}}}{\eta(\tau)^{N(N+1)}d_{N-1}}
\partial P_{r-\frac{N}{2}}(\tau)\Bigl|_{\left\{e^t=1\big\vert t\in \frac{1}{N}A_{N-1}\right\}}.
\end{split}
\end{equation}

\end{theorem}
\begin{remarks}

1.
  Using \eqref{eq:charbg} and \eqref{eq:charh}, Theorem \ref{FourierCoeffsThomas} implies the character decomposition
 \begin{equation}\nonumber
\begin{split}
  \text{ch}[\mathcal S(N)](z,\dots,z;\tau) &= \sum_{r\in\Z}  \text{ch}\left[\mathcal F_{\frac r{\sqrt{N}}}(z;\tau)\right]
  \text{ch}[\mathcal B_r](\tau), \qquad
\end{split}
\end{equation}
where
\begin{equation}\nonumber
\begin{split}
 \text{ch}[\mathcal B_r](\tau) :&=  \frac{\mathrm{sign}(r)}{\eta(\tau)^{N^2-1}d_{N-1}} 
\sum_{t\in T_r} \prod_{\alpha\in{\Delta_0^+}} (t|\alpha)q^{\frac{t^2}{2(N-1)}}.
\end{split}
\end{equation}
The $\text{ch}[\mathcal B_r](\tau)$ are then characters of $L_{-1}(\Sl(N))$.

2.
 The proof of the theorem uses the denominator identity of both $\Sl(N|1)$ and $\widehat{\Sl}(N|1)$ as
well as Weyl's character formula for $\Sl(N)$.

3.
The case $N=1$ follows from the denominator identity of $\widehat{\mathrm{g}\ell}(1|1)$ (see Example \ref{prop:N=1}). In this case, the Fourier coefficients relate to the characters of a well-known logarithmic conformal field theory, the $\mathcal W(2,3)$-algebra of central charge $-2$. 
The modularity of the coefficients has been studied from a different perspective in \cite{CM}.

\end{remarks}

The second approach is based on a generalization of a deep identity of Atkin and Garvan. To state it, we first recall the rank and crank generating functions (whose combinatorial meanings are not needed in this paper), which arise in many contexts and in particular give combinatorial explanations of Ramanujan's congruences (for example see \cite{AndrewsGarvanCrank,AtkinSDRank,Dyson}). Specifically, the generating functions are given as follows: 

\[\mathcal R(\zeta;q):=\frac{(1-z)}{(q)_{\infty}}\sum_{n\in\Z}\frac{(-1)^nq^{\frac n2(3n+1)}}{1-\zeta q^n},\]

\begin{equation}\label{crankproduct}
\hspace{-0.7in}\mathcal C(\zeta;q):=\frac{(q)_{\infty}}{(\zeta q)_{\infty}(\zeta^{-1}q)_{\infty}}.
\end{equation}

We also need the normalized versions
\[\mathcal R^*(\zeta;q):=\frac{\zeta^{\frac12}q^{-\frac1{24}}\mathcal R(\zeta;q)}{1-\zeta},\quad\quad\quad\quad\mathcal C^*(\zeta;q):=\frac{\zeta^{\frac12}q^{-\frac1{24}}\mathcal C(\zeta;q)}{1-\zeta}.\]

\noindent Note that $\phi_N$ is essentially the $N$-th power of $\mathcal C^*$ as for $N\in\N$ we have \begin{equation}\label{PhiCrank}\phi_{N}(z;\tau)=i^N\eta(\tau)^{-2N}\mathcal C^*(\zeta;q)^N.\end{equation}
The simplest case of our decomposition relies on the fact that $\mathcal C^*$, and thus $\phi_1$, is essentially an Appell-Lerch sum thanks to the following classical partial fraction expansion (for example, see Theorem 1.4 of \cite{ZwegersRankCrankPDE}).

\begin{equation}\label{CrankPartialFrac}
\mathcal C^*(\zeta;q)=\frac{\zeta^{\frac12}}{\eta(\tau)}\sum_{n\in\Z}\frac{(-1)^nq^{\frac{n(n+1)}2}}{1-\zeta q^n}.
\end{equation}

For the cube of the crank generating function, Atkin and Garvan \cite{AG} proved the following rank-crank PDE which is very useful in establishing congruences and relations between the moments of the rank and crank generating functions:

\begin{equation}\label{rankcrankPDE} 2\eta(\tau)^2\mathcal C^*(\zeta;q)^3=\left(6\mathcal D_q+\mathcal D_\zeta^2\right)\mathcal R^*(\zeta;q).\end{equation}
Here and throughout $\mathcal D_x:=x\frac{\partial}{\partial x}$. 
Note that this gives a description of $\phi_3$ in terms of Appell-Lerch sums by (\ref{PhiCrank}).

Zwegers \cite{ZwegersRankCrankPDE} nicely generalized (\ref{rankcrankPDE}) for arbitrary odd powers of the crank generating function using the theory of elliptic forms. For similar results using another clever proof, see also the paper of Chan, Dixit, and Garvan \cite{ChanDixitGarvan}.

In this paper, we prove a new family of analogous PDEs which are of independent interest. Moreover, we package Zwegers' family of PDEs in a way which illuminates their structure coming from negative index Jacobi forms. To describe this, we need the Appell-Lerch sums 
\begin{equation}\label{defineFN}
F_N(z,u;\tau):=\zeta^{\frac N2}w^{\frac N2}\sum_{n\in\Z}\frac{(-w)^{Nn}q^{\frac N2n(n+1)}}{1-\zeta wq^n},
\end{equation}
where $w:=e^{2\pi i u}$. We note that these Appell-Lerch sums are similar to the functions $f_z(u;\tau)$ considered in Chapter 3 of \cite{ZwegersThesis}, which transform as a Jacobi form in $u$ and as a ``mock Jacobi form'' in $z$. We also require the Laurent coefficients of $\phi_{M,N}(z;\tau)$ at $z=0$:

\begin{equation}\label{LaurentCoeffsDefn}\phi_{M,N}(z;\tau)=\frac{D_N(\tau)}{(2\pi iz)^N}+\frac{D_{N-2}(\tau)}{(2\pi iz)^{N-2}}+\ldots+O(1).\end{equation}

Note that only even or odd Laurent coefficients occur, depending on the parity of $N$, since $\vartheta(-z;\tau)=-\vartheta(z;\tau)$. 
It is not hard to see that the coefficients $D_j$ are quasimodular forms. Explicitly, they can be computed quickly in terms of the usual Eisenstein series \[G_k(\tau):=-\frac{B_k}{2k}+\sum_{n\geq1}\sigma_{k-1}(n)q^n,\] where $\sigma_k(n):=\sum_{d|n}d^k$ and $B_k$ is the usual $k$-th Bernoulli number. Specifically,  it easily follows from the Jacobi triple product formula that
\[\vartheta(z;\tau)=-2\pi z\eta^3(\tau)\operatorname{exp}\left(-2\sum_{k\geq1}G_{2k}(\tau)\frac{(2\pi iz)^{2k}}{(2k)!}\right).\]
The following result puts Zwegers family of PDEs as well as our new family of PDEs into a common framework. Setting \begin{equation*}
\delta_e:=\begin{cases}0&\text{ if } N\in2\N-1,\\ 1&\text{ if }N\in2\N,\end{cases}
\end{equation*}
we find: 
\begin{theorem}\label{RankCrankEven}
For any $N\in\N$, $M\in\N_0$, we have 
\[\phi_{2M,N+2M}(z;\tau)=(-1)^{1+\delta_e}\sum_{j=0}^{\frac{N-1-\delta_e}2+M}\frac{D_{2j+\delta_e+1}(\tau)}{(2j+\delta_e)!}\mathcal D_w^{2j+\delta_e}\left(F_N(z,u;\tau)\right)\big|_{w=1}.\]
\end{theorem}

\begin{remarks}

1. Note that Theorem \ref{RankCrankEven} is more explicit than Zwegers' rank-crank type PDEs as it gives the modular coefficients of the PDEs directly from the structure of the Jacobi form $\phi_{M,N}$. Chan, Dixit, and Garvan also remarked that it would be interesting to find such an explicit expression for the quasimodular forms in the decomposition in that case.

2. It would be interesting to find a Lie theoretic interpretation of the decomposition in Theorem \ref{RankCrankEven}.

\end{remarks}

Armed with the decomposition in Theorem \ref{RankCrankEven} into Appell-Lerch sums, we can easily pick off the Fourier coefficients $\chi(M,N,r;\tau)$ and write them in terms of the Laurent coefficients of $\phi_{M,N}$ and certain partial theta functions 
\[\partialtheta_{\frac12+\delta_e}(N,r;\tau):=\sum_{n\geq0}(-1)^{Nn}\left(n+\frac rN\right)^{\delta_e}q^{\frac N2\left(n+\frac r N\right)^2}.\]
Specifically, if we let
 \[\rho(r):=\begin{cases}r&\text{ if }r\geq\frac N2,\\ N-r&\text{ if }r<\frac N2,\end{cases}\]

\noindent then the Fourier coefficients of $\phi_{M,N}$ are as follows.

\begin{theorem}\label{EvenDecomp}
For  any $N\in\N$, $M\in\N_0$, $r\in\frac{N}2+\Z$, $0\leq\operatorname{Im}(z)<\operatorname{Im}(\tau)$, we have 
\[\chi(2M,N+2M,r;\tau)=(-1)^{1+\delta_e}q^{-\frac{r^2}{2N}}\sum_{j=0}^{\frac{N-1-\delta_e}2+M}\frac{D_{2j+\delta_e+1}(\tau)}{(2j+\delta_e)!}N^{j+\delta_e}2^{j}\mathcal D_q^{j}\left(\Theta_{\frac12+\delta_e}(N,\rho(r);\tau)\right).\]
\end{theorem}

If $N>1$ is odd, these partial theta functions fit into the pioneering work of Folsom, Ono, and Rhoades \cite{FOR} which gives startling relations between the asymptotic expansions of the rank and crank generating functions, generalizing and proving beautiful formulas of Ramanujan. Their work shows that $\Theta_{\frac12}(N,r;\tau)$ is a strong quantum modular form for odd $N>1$. Although their theorem does not directly apply for $N=1$, in this case we essentially obtain an eta quotient which is trivially a quantum modular form at cusps where it vanishes. 

For even $N$, both the hypergeometric representations used to determine quantum sets and the proof of quantum modularity are not applicable. Here we use the innovative approach of Lawrence and Zagier \cite{Lawrence-Zagier} to study quantum modularity properties (see also \cite{ZagierVassiliev}). A key ingredient in our investigation is a beautiful identity of Warnaar \cite{Warnaar} which relates certain partial and false theta functions (see (\ref{WarnaarFormula})). Our main result for studying quantum modularity for even $N$ is the following, which gives a new family of quantum modular forms. 

\begin{theorem}\label{QuantumWt32}
For any $N\in 2\N$ and $r\in\Z$, $\Theta_{\frac32}(N,r;\tau)$ is a strong quantum modular form with quantum set $\widehat{Q}_{N, r}$ (defined in (\ref{QuantumSetDefn})) on $\Gamma_1(2N)$, multiplier system $\chi_{r}$ (defined in (\ref{ChiMultiplierDefn})), and weight $\frac32$.  \end{theorem}
\begin{remarks}

1.
More details about the specific quantum modular properties can be found in the proof of Theorem \ref{QuantumWt32} in Section \ref{ProofQuantumModular}.

2.
More generally, using Proposition 3 of \cite{Zagier}, our proof of Theorem \ref{QuantumWt32} shows that $\Theta_{\frac32}(N,r;\tau)$ has modularity properties on all of $\Q$. For this, we note that although the function is not defined on all of $\Q$, it has a well-defined asymptotic expansion at all points in $\Q$. This expansion still agrees with the non-holomorphic Eichler integral on the lower half plane (see Section 6), so one could say that $\Theta_{\frac32}(N,r;\tau)$ is a quantum modular form on $\Q$ if we allow ``poles'' at certain points in $\Q$.
\end{remarks}

The paper is organized as follows. In Sections 2 and 3, we review the necessary notation and basic objects from Lie theory, Jacobi forms, and quantum modular forms. We give our first proof of the decomposition using Lie theory in Section 4 and our second proof using an analogue of the rank-crank PDE in Section 5.  We conclude by describing the quantum modular properties of $\Theta_{\frac12+\nu}(N,r;\tau)$ in Section 6. 

\section*{Acknowledgements}

The authors began jointly discussing this work at the Mock Modular forms, Moonshine, and String Theory conference in Stony Brook, August 2013 and are grateful for the well-organized conference and great working environment. The authors would also like to thank Sameer Murthy for pointing out the relation of our work to partition functions arising from physics, as well as Drazen Adamovic, Jan Manschot, Robert Rhoades, Jeffrey Harvey, and Sander Zwegers for helpful conversations. 
%The research of the first author was supported by the Alfried Krupp Prize for Young University Teachers of the Krupp foundation and has received funding from the European Research Council under the European %Union's Seventh Framework Programme (FP/2007-2013) / ERC Grant agreement n. 335220 - AQSER. The authors began jointly discussing this work at the Mock modular forms, moonshine, and string theory %conference in Stony Brook in August, 2013 and are grateful to the organizers for providing a pleasant conference and work environment. The authors would also like to thank Sameer Murthy for pointing out the %elation of our work to partition functions arising from physics, and also Sander Zwegers for helpful conversations. Also thank Jan Manschot and Adamovic and Rhoades and Harvey.}

\section{Preliminaries on Lie super algebras and character identities}

In this section, we recall some known facts of the affine Lie superalgebra $\widehat{\Sl}(N|1)$, following \cite{KW1}, as well as the finite dimensional Lie algebra $\Sl(N)$ using \cite{Hum}. 

\subsection{The Lie super algebra $\Sl (N+1|1)$}
In this subsection the Lie super algebra $\Sl (N+1|1)$ and its root system are defined.

The even subalgebra of the Lie super algebra $\Sl(N+1|1)$ 
is g$\ell(N+1)$ and the odd part decomposes into the standard representation of the even subalgebra and its conjugate. 
In order to define the Lie super algebra, it is convenient to first introduce its root system. 
It lies in the lattice
\[
 L_N := \Z \varepsilon_1 \oplus \dots\oplus \Z \varepsilon_{N+1} \oplus \Z\delta
\]
with bilinear form
\[
(\varepsilon_j|\varepsilon_k):=\delta_{j,k},\quad (\varepsilon_j|\delta):=0,\quad (\delta|\delta):=-1.
\]
Thus, its signature is $(N+1,1)$.
The set of roots is $\Delta = \Delta_0 \cup \Delta_1\subset L_N$, where the
set of even roots (resp. odd roots) is denoted by $\Delta_0$ (resp. $\Delta_1$). 
They are
\begin{equation}\label{eq:rootspaces}
\begin{split}
\Delta_0 &:= \{ \varepsilon_j-\varepsilon_k  |  1\leq j,k \leq N+1, j\neq k \},\qquad 
 \Delta_1 := \{ \varepsilon_j-\delta, \delta-\varepsilon_j  |  1\leq j \leq N+1 \}.
\end{split}
\end{equation}
It is useful to split these sets into positive and negative subroot spaces, where
\begin{equation}\nonumber
\begin{split}
\Delta_0^+ &:= \{ \varepsilon_j-\varepsilon_k  |  1\leq j < k \leq N+1 \}, \quad
 \Delta_1^+ := \{ \varepsilon_j-\delta  |  1\leq j \leq N+1 \}, \\
\Delta^+ &:= \Delta_0^+ \cup \Delta_1^+, \qquad \Delta:= \Delta^+\cup\Delta^-.
\end{split}
\end{equation}
A distinguished system of simple positive roots is then chosen to be
\begin{equation}\nonumber
\Pi := \{ \alpha_j=\varepsilon_j-\varepsilon_{j+1}, \beta=\varepsilon_{N+1}-\delta | 1\leq j\leq N\}.
\end{equation}
The $\alpha_j$ are even roots and $\beta$ is the only distinguished odd simple root. 
The inner products of simple positive roots are
\[ (\alpha_j|\alpha_k) =
\begin{cases} 2 &\quad\text{if} \ j=k, \\
-1 &\quad \text{if}\ j=k\pm 1, \\
0 & \quad \text{otherwise,} \end{cases} \qquad (\beta|\alpha_j)=-\delta_{j,N}, \qquad (\beta|\beta)=0.
\]
Hence $\beta$ is an isotropic root. 
Simple even roots generate the even root lattice
\begin{equation}\nonumber
A_N := \Z \alpha_1\oplus \dots \oplus \Z\alpha_N.
\end{equation}
Its dual lattice is
\begin{equation}\nonumber
A_N':= \Z \lambda_1\oplus \dots \oplus \Z\lambda_N,
\end{equation}
where the inner product of the fundamental weights $\lambda_j$ with simple roots is $(\lambda_j|\alpha_k)=\delta_{j,k}$ and $(\lambda_j|\delta)=0$.
Roots and weights act on the Cartan subalgebra, which is 
\[
\mathfrak h := \bigoplus_{\alpha\in\Pi} \C h_\alpha = \mathfrak h_0 \oplus \C h_\beta,\quad
\mathfrak h_0 := \C h_{\alpha_1}\oplus \dots\oplus \C h_{\alpha_N},
\]
with basis $\{h_\alpha\}$ parameterized by simple positive roots, and $\mathfrak h_0$ the Cartan subalgebra of $\Sl(N+1)$. 
The fundamental weights $\lambda_j$ are identified with elements of the dual $\mathfrak h^*_0$ of $\mathfrak h_0$ via $\lambda_j(h_{\alpha_k})=\delta_{j,k}$. 
A bilinear form $( \ \ , \ \ )$ on $\mathfrak h$ is induced from the form on its dual space via 
\[
\left(h_\alpha, h_{\alpha'}\right) := \left(\alpha|\alpha'\right).
\]

We remark that the Lie superalgebra $\Sl(N+1|1)$  is then the $\Z/2\Z$-graded algebra generated by $\{ h_\alpha, e^\pm_\alpha | \alpha \in \Pi\}$ subject to the Serre-Chevalley relations (\ref{Serre-Chevalley}) and the graded Jacobi identity. 
The parity of $h_\alpha$ and $e^\pm_{\alpha_j}$ is even, while the $e^\pm_\beta$ are odd. 
We denote the graded anti-symmetric bracket by $[ \ \ , \ \ ]\colon  \Sl(N+1|1)\times  \Sl(N+1|1)\rightarrow \Sl(N+1|1)$. 
Then the Serre-Chevalley relations of the algebra are
\begin{equation}\label{Serre-Chevalley}
\begin{split}
[h_{\alpha}, h_{\alpha'}]=0,\quad \left[h_{\alpha}, e^\pm_{\alpha'}\right]=\pm  (\alpha|\alpha')e^\pm_{\alpha},
\quad \left[e^+_{\alpha},e^-_{\alpha'}\right]=\delta_{\alpha,\alpha'}h_{\alpha},\quad
(\text{ad}\, e^\pm_{\alpha})^{1-(\alpha|\alpha')}e^\pm_{\alpha'}= 0
\end{split}
\end{equation}
for all $\alpha, \alpha' \in \Pi$ and $\alpha\neq \alpha'$ in the last equation. The bilinear form $(\ \ , \ \ )$ on $\mathfrak h$ can be extended 
to an invariant non-degenerate graded symmetric form on $\Sl(N+1|1)$, which we also denote by $(\ \ , \ \ )$.

\subsection{The even Weyl group and denominator identity of $\Sl(N+1|1)$}
We now introduce the even Weyl group and the denominator identity of the Lie super algebra $\Sl(N+1|1)$.

For this, we first need to define the Weyl vector $\rho$. It is the difference of the even Weyl vector $\rho_0$ and the odd one $\rho_1$; namely
\begin{equation}\nonumber
\begin{split}
\rho_0 &:= \frac{1}{2} \sum_{\alpha\in\Delta_0^+} \alpha = \frac{1}{2} \sum_{j=1}^{N+1}(N+2-2j)\varepsilon_j,\qquad
\rho_1 := \frac{1}{2} \sum_{\alpha\in\Delta_1^+} \alpha = \frac{1}{2}\left(-(N+1)\delta+ \sum_{j=1}^{N+1}\varepsilon_j \right),\\
\rho &:= \rho_0-\rho_1 = \frac{1}{2} \left( (N+1)\delta+ \sum_{j=1}^{N+1}(N+1-2j)\varepsilon_j\right).
\end{split}
\end{equation}

The group of even Weyl reflections $W^\sharp$ acts on the dual of the even root lattice, $A_N'$,  and
 is generated by $\sigma_j, j=1,...,N$ defined by
\begin{equation}\nonumber
\begin{split}
\sigma_j\colon A_N' \rightarrow A_N',\qquad  \sigma_j(\lambda):=\lambda -(\alpha_j|\lambda)\alpha_j.
\end{split}
\end{equation}
This action naturally extends to the lattice $L_N$ via $\sigma_j(\delta)=0$ and
\begin{equation}\nonumber
\begin{split}
\sigma_{j}(\varepsilon_k)&:=\varepsilon_k-(\varepsilon_k|\alpha_j)\alpha_j %= \varepsilon_k-(\lambda_k-\lambda_{j-1}|\alpha_j)(\varepsilon_j-\varepsilon_{j+1}) 
= \varepsilon_k-(\delta_{j,k}-\delta_{j+1,k})(\varepsilon_j-\varepsilon_{j+1}) = 
\begin{cases}
\varepsilon_{j+1} & j=k, \\ \varepsilon_{j} & j=k-1, \\
\varepsilon_k & \text{otherwise}.
\end{cases}
\end{split}
\end{equation}
Hence the even Weyl group $W^\sharp$ is just the group $S_{N+1}$ permuting the $\varepsilon_j$. Orthonormality of the $\varepsilon_j$ implies that the even Weyl group preserves the bilinear form $( \ \ | \ \ )$.
Following \cite{KW1} we define

\begin{defn}

A \begin{bf}{regular exponential function}\end{bf} on $A_N'$ is a finite linear combination of exponentials of the form $e^\lambda$ for $\lambda\in A_N'$. 
A \begin{bf} rational exponential function \end{bf} is the quotient $A/B$ of two regular exponential functions $A$ and $B\neq 0$. 
The even Weyl group $W^\sharp$ acts on the field of these functions as $e^\lambda \mapsto e^{w(\lambda)}$ for any $w\in W^\sharp$.
The \begin{bf} Weyl denominator \end{bf} of $\Sl(N+1|1)$ is the rational exponential function
\begin{equation}\nonumber
R = \frac{\prod\limits_{\alpha\in\Delta_0^+}(1-e^{-\alpha})}{\prod\limits_{\alpha\in\Delta_1^+}(1+e^{-\alpha})}.
\end{equation}

\end{defn}

We saw that the even Weyl group $W^\sharp$ is just $S_{N+1}$, the signum of an element $w$ in $W^\sharp$ is
$\sigma(w):=(-1)^n$ if $w$ can be written as a composition of $n$ transpositions. 
Theorem 2.1 of \cite{KW1} applied to our situation gives

\begin{lemma}\label{lem:findenom}

The Weyl denominator of the Lie super algebra $\Sl(N+1|1)$ is 
\begin{equation}\nonumber
e^\rho R = \sum_{w\in W^\sharp}\sigma(w) w\left( \frac{e^\rho}{\left(1-e^{\delta-\varepsilon_{N+1}}\right)}\right).
\end{equation}

\end{lemma}

\subsection{The denominator identity of the affine Lie super algebra $\widehat{\Sl}(N+1|1)$}
We turn our focus to the affinization of $\Sl(N+1|1)$, that is
\[
 \widehat\Sl(N+1|1) := \C\left[t,t^{-1}\right]\otimes \Sl(N+1|1) \oplus \C C \oplus \C d
\]
with bracket
\[
 [t^n\otimes x, t^m\otimes y]:=t^{n+m}\otimes[x,y]+n\delta_{n+m,0}(x,y)C, \ \
 [d,t^n\otimes x]:= nt^n\otimes x, \ \  [C,t^n\otimes x]:=[C,d]:=0
\]
for all $x,y\in\Sl(N+1|1)$ and $n,m\in\Z$. 
The Cartan subalgebra extends to its affine counterpart
\begin{equation}\nonumber
\mathfrak {\widehat h}:=\mathfrak h \oplus \C d\oplus \C C
\end{equation}
and the bilinear form extends as
\begin{equation}\nonumber
(t^n\otimes x, t^m\otimes y):=(x,y)\delta_{n+m,0}, \ \
(C,d):=1, \ \
(C,C):=(d,d):=(C,t^n\otimes x):=(d,t^n\otimes x):=0.
\end{equation}
We identify $C$ and $d$ with linear functionals on $\mathfrak h$ using the bilinear form $( \ \ , \ \ )$ and extend $A_N'$ to  
\begin{equation}\nonumber
\widehat{A}_N' := A_N'\oplus \Z d\oplus \Z C.
\end{equation}
The bilinear form extends as
\begin{equation}\nonumber
(C|d):=1\quad\text{and}\quad
(C|C):=(d|d)=(C|\lambda)=(d|\lambda)=0 \quad \forall \lambda \in A_N'.
\end{equation}
The lattice $A_N\subset A_N'$ is then also a sublattice of $\widehat{A}_N'$.
The affine Weyl vector is
\begin{equation}\nonumber
\widehat{\rho}:=\rho+Nd.
\end{equation}
Note that $N$ is the dual Coxeter number of $\Sl(N+1|1)$. 
For $\alpha\in A_N$, we define 
\begin{equation}\label{eq:transl}
t_\alpha\colon \widehat{A}_N' \rightarrow \widehat{A}_N',\qquad
\lambda \mapsto \lambda +(\lambda|C)\alpha -\left( (\lambda|\alpha)+\frac{1}{2}(\alpha|\alpha)(\lambda|C)\right) C.
\end{equation}
The group of even Weyl translations is $\{ t_\alpha| \alpha \in A_N\}$. 
Conjugation by a Weyl rotation gives for any $w\in W^\sharp,$ $\alpha\in A_N$
\begin{equation}\label{eq:conj}
w\circ t_\alpha\circ w^{-1}= t_{w(\alpha)}.  
\end{equation}
Let 
\begin{equation*}
Y:= \left\{ h\in\widehat{\mathfrak h} | \text{Re}(C(h))>0\right\}
\end{equation*}
be the domain of all elements in $\widehat{\mathfrak h}$ on which the action of $C$ has positive real part. 
Let $\widehat{\mathcal F}$ be the field of meromorphic functions on $Y$ and define $\qq:=e^{-C}$. Thus $|\qq(y)|<1$ for all $y$ in $Y$. 
Any element $\lambda$ of $L'$ extends to a linear function on $\widehat{\mathfrak h}^*$ by defining 
$\lambda(C)=\lambda(d)=0$. In this way rational exponential functions on $L'$ embed in $\widehat{ \mathcal F}$. 
\begin{defn}
The {\bf denominator} of $\widehat{\Sl}(N+1|1)$ is
\begin{equation}\nonumber
\widehat{R} = R \prod_{j=1}^\infty \left(1-\qq^j\right)^{n+1}\prod_{\alpha\in\Delta_0}\left(1-\qq^je^\alpha\right) \prod_{\alpha\in \Delta_1}\left(1+\qq^je^\alpha\right)^{-1}.
\end{equation}
\end{defn}
We need Theorem 4.1 of \cite{KW1}, which states
\begin{lemma}\label{lem:affdenom}

The denominator of $\widehat{\Sl}(N+1|1)$ satisfies
\begin{equation}\nonumber
e^{\widehat{\rho}} \widehat{R} = \sum_{\alpha\in A_N}t_\alpha\left(e^{\widehat{\rho}}R\right).
\end{equation}
\end{lemma}

\subsection{The Weyl character formula of $\Sl(N+1)$}

We also require a well-known variant of the dimension formula, which itself is a corollary of the famous character formula of Weyl \cite{W}.
Let $\lambda= m_1\lambda_1+\dots+m_N\lambda_N$ be a dominant weight of $\Sl(N+1)$; that is, all $m_j$ are natural numbers. Letting $V_\lambda$ be the corresponding irreducible highest-weight module, then the character formula is in our notation
\begin{equation}\label{eq:weylchar}
\text{ch}[V_\lambda]=\frac{\sum\limits_{w\in W^\sharp}\sigma(w) e^{w(\lambda+\rho_0)}}{e^{\rho_0}\prod\limits_{\alpha\in\Delta_0^+}(1-e^{-\alpha})}.
\end{equation}
Since $\rho_1$ is $W^\sharp$ invariant we can replace $\rho_0$ by $\rho$ in this formula .
Let $m\in\N$ and let $v$ be the linear map from the regular exponential functions on $\frac{1}{m}A_N'$ to the complex numbers defined by
$v(e^\lambda)=1$ for every $\lambda\in \frac{1}{m}A_N'$. 
Let $V_\lambda$ be the irreducible finite dimensional highest-weight module of highest-weight $\lambda$. Hence $v(\text{ch}[V_\lambda])$ is just the dimension of this module. 
The application of $v$ to both nominator and denominator of the character formula \eqref{eq:weylchar} vanishes, but the quotient is finite. Using \eqref{eq:partial}, we find \cite{Hum}
\begin{equation}\nonumber
v(\text{ch}[V_\lambda]) =\frac{v \left( \partial\sum\limits_{w\in W^\sharp}\sigma(w) e^{w(\lambda+\rho_0)}\right)}{v\left(\partial\, e^{\rho_0}\prod\limits_{\alpha\in\Delta_0^+}(1-e^{-\alpha})\right)} =\frac{\prod\limits_{\alpha\in\Delta_0^+}(\lambda+\rho_0|\alpha)}{\prod\limits_{\alpha\in\Delta_0^+}(\rho_0|\alpha)}.
\end{equation}
Note that this is Weyl's character formula for irreducible finite-dimensional highest-weight modules. The second equality also holds 
if we replace $\lambda+\rho_0$ by $zw(\lambda+\rho_0)$ for any complex number $z$ and any $w$ in $W^\sharp$.   
\begin{defn}
If $m\in\N$ and $\mu$ in $\frac{1}{m}A_N'$,
then $\mathbf{v_\mathbf{\mu}}$ is the rational exponential function
\begin{equation}\label{eq:vmu}
\mathrm{v}_\mu= \frac{\sum\limits_{w\in W^\sharp}\sigma(w) e^{w(\mu)}}{e^{\rho_0}\prod\limits_{\alpha\in\Delta_0^+}(1-e^{-\alpha})}.
\end{equation}
\end{defn}
Note that if $\mu-\rho_0$ is dominant, then this is just the character of the irreducible highest-weight module of highest-weight $\mu-\rho_0$. 
We now closely follow the argument of the proof of the dimension formula of  \cite{Hum}. 
\begin{lemma}\label{lemma:vvmu}
If $m\in\N$ and $\mu$ in $\frac{1}{m}A_N'$, then 
\[
 v(v_\mu)=\frac{1}{d_N}\prod\limits_{\alpha\in\Delta_0^+}(\mu|\alpha)=
 \frac{v(\partial e^{\mu})}{d_N}.
\]
\end{lemma}
\begin{proof}

Using the explicit description of the positive even roots in \eqref{eq:rootspaces}, it is easy to compute 
 \begin{equation}\nonumber
\prod\limits_{\alpha\in\Delta_0^+}(\rho_0|\alpha)= \prod_{n=1}^Nn! =d_N.
\end{equation}
For an arbitrary weight $\mu\in A_N'$, there exists a unique $w\in W^\sharp$ such that $w(\mu+\rho_0)-\rho_0$ is dominant. Letting $\ell(w)$ be the number of positive roots that are mapped to negative ones by $w$, then $(-1)^{\ell(w)}=\sigma(w)$ (see \cite{Hum}). Then using that the even Weyl group respects the bilinear form.
\begin{equation}\nonumber
\begin{split}
v\left(v_{\frac{\mu}{M}}\right) &= 
v\left(\frac{\sum\limits_{w\in W^\sharp}\sigma(w) e^{w\left(\frac{\mu}{M}\right)}}{e^{\rho_0}\prod\limits_{\alpha\in\Delta_0^+}(1-e^{-\alpha})}\right) = 
(-1)^{\ell(w)} \frac{\prod\limits_{\alpha\in\Delta_0^+}\left(w\left(\frac{\mu}{M}\right)\Big|\,\alpha\right)}{\prod\limits_{\alpha\in\Delta_0^+}(\rho_0|\alpha)}  =
\frac{(-1)^{\ell(w)}}{d_N} \prod\limits_{\alpha\in\Delta_0^+}\left(\frac{\mu}{M}\Big|w^{-1}(\alpha)\right)\\
&= \frac{(-1)^{\ell(w)+\ell(w^{-1})}}{d_N}\prod\limits_{\alpha\in\Delta_0^+}\left(\frac{\mu}{M}\Big|\,\alpha\right)
= \frac{1}{d_N}\prod\limits_{\alpha\in\Delta_0^+}\left(\frac{\mu}{M}\Big|\,\alpha\right)= \frac{v\left(\partial e^{\frac{\mu}{M}}\right)}{d_N} .
\end{split}
\end{equation}
 \end{proof}

\section{Basic facts on Jacobi forms and quantum modular forms}
\subsection{Jacobi Forms}

Here we recall some special Jacobi forms and previous work on Fourier coefficients of Jacobi forms. Jacobi forms are functions from $\C\times\H\rightarrow\C$ which satisfy both an elliptic and a modular transformation law. For the precise definition and basic facts on Jacobi forms we refer the reader to \cite{EichlerZagier}. In this paper, we are particularly interested in the classical Jacobi theta function, defined in (\ref{JacobiThetaDefn}). The following transformation laws are well-known (for example, see \cite{Rademacher} (80.31) and (80.8)).
\begin{lemma}
For $\lambda,\mu\in\Z$ and $\gamma=\left(\begin{smallmatrix}a&b\\c&d\end{smallmatrix}\right)\in\SL_2(\Z)$, we have that 

\[\vartheta\left(\frac{z}{c\tau+d};\frac{a\tau+b}{c\tau+d}\right)=\psi^3(\gamma)(c\tau+d)^{\frac12}e^{\frac{\pi i c z^2}{c\tau+d}}\vartheta(z;\tau),\]
\[\vartheta(z+\lambda\tau+\mu;\tau)=(-1)^{\lambda+\mu}q^{-\frac{\lambda^2}2}e^{-2\pi i \lambda z}\vartheta(z;\tau),\]
where $\psi(\gamma)$ is the multiplier arising in the transformation law of Dedekind's eta function.
\end{lemma}
We also require the following theta functions of weight $\frac12+\nu$ defined for $r\in\Z$, $\nu\in\{0,1\}$
\[\vartheta_{\frac12+\nu}(N,r;\tau):=\sum_{n\in\Z}(-1)^{n N}\left(n+\frac{r}{N}-\frac12 \right)^{\nu}q^{\frac N2\left(n+\frac{r}{N} - \frac12\right)^2}.\]
We define for convenience the following shifted versions when $N\in2\N$
\begin{equation}\label{ThetaTilde}\widetilde{\vartheta}_{\frac12+\nu}(N,r;\tau):=\vartheta_{\frac12+\nu}\left(N,r+\frac N2;\tau\right).\end{equation}
It is trivial to show the following identities:
\begin{equation}\label{ThetaShiftN}
\widetilde{\vartheta}_{\frac12+\nu}(N,r+N;\tau)=\widetilde{\vartheta}_{\frac12+\nu}(N,r;\tau),
\end{equation}
\begin{equation}\label{ThetaNegateN}
\widetilde{\vartheta}_{\frac12+\nu}(N,-r;\tau)=(-1)^{\nu}\widetilde{\vartheta}_{\frac12+\nu}(N,r;\tau),
\end{equation}
\begin{equation}\label{ThetaSpecial}
\vartheta_{\frac12}(1,0;\tau)=\widetilde{\vartheta}_{\frac32}(2,r;\tau)=0.
\end{equation}

In Sections 5 and 6, we need the following modular transformations, which can be derived as special cases of the transformation formulas for the theta functions of Shimura \cite{Shimura}. \begin{proposition}\label{ThetaVectorTrans} If $N\in 2\N$ and $r\in\Z$, then we have:   
\[\widetilde{\vartheta}_{\frac32}\left(N,r;\tau+1\right)=e\left(\frac{r^2}{2N}\right)\widetilde{\vartheta}_{\frac32}(N,r;\tau),\]
 \[\widetilde{\vartheta}_{\frac32}\left(N,r;-\frac1{\tau}\right)=\frac{2}{\sqrt{N}} (-i \tau)^{\frac32}
\sum_{k=1}^{\frac N2-1}\sin\left(\frac{2\pi kr}{N}\right)\widetilde{\vartheta}_{\frac32}(N,k;\tau).\]
Moreover,  for $\nu\in\{0,1\}$, $\widetilde{\vartheta}_{\frac12+\nu}(N,r;\tau)$ is a modular form of weight $\frac12+\nu$ on $\Gamma_1(2N)$ with multiplier

\begin{equation}\label{ChiMultiplierDefn}\chi_r\begin{pmatrix}a&b\\c&d\end{pmatrix}:=\begin{dcases}e\left(\frac{br^2}{2N}\right) &\text{ if }c=0,\\e\left(\frac{br^2}{2N}\right)\left(\frac {2Nc}d\right)&\text{ if } c\neq0.\end{dcases}\end{equation}

\end{proposition}

\noindent We remark that in Proposition \ref{ThetaVectorTrans}, $\vartheta_{\frac12+\nu}(N,r;\tau)$ are actually modular forms on a slightly larger congruence subgroup, but we have chosen to use $\Gamma_1(2N)$ for ease of exposition.

We next recall the structure of Fourier coefficients of positive index Jacobi forms for comparison. It is well known that holomorphic Jacobi forms have a theta decomposition involving the functions \[\vartheta_{m,b}(z;\tau):=\displaystyle\sum_{\substack{\lambda\in\Z\\ \lambda\equiv b\pmod{2m}}}e^{\frac{\pi i \lambda^2\tau}{2m}+2\pi i \lambda z}.\] The components of this decomposition are classical (vector-valued) modular forms \cite{EichlerZagier}. The Fourier coefficients of meromorphic Jacobi forms of positive index are also understood.
 Specifically, in \cite{BringmannFolsomKacWaki}, Folsom and the first author, building on illuminating work of Dabholkar, Murthy, and Zagier \cite{DMZ} and Zwegers \cite{ZwegersThesis}, considered the Kac-Wakimoto character of level $(M,N)$ with $M>N$, $M,N\in2\N$, which essentially corresponds to the meromorphic Jacobi form $\phi_{M,N}$ (the general case with $M>N$ is considered in \cite{Ol}). These Kac-Wakimoto characters have a decomposition into a finite and a polar part, where the finite part has a theta decomposition similar to that of holomorphic Jacobi forms (but involving mock modular forms), and where the polar part is

\begin{equation}\label{KacWakiPositiveIndex}\varphi^P(z;\tau):=-\sum_{j=1}^{\frac N2}\frac{D_{2j}(\tau)}{(2\pi i)^{2j-1}(2j-1)!}\frac{\partial^{2j-1}}{\partial u^{2j-1}}\left(e^{\pi i(M-N)u} \zeta^{-\frac{M-N}2}F_{M-N}\left(\frac{\tau}2-u,z-\frac{\tau}2;\tau\right)\right)\bigg|_{u=0}.\end{equation}
Here $D_{j}$ is the $j$-th Laurent coefficient of the level $(M,N)$ Kac-Wakimoto character. 
Thus we see that our functions $\phi_{M,N}$ have decompositions which are strikingly similar to the decompositions of positive index Jacobi forms, although in our case there are no associated ``finite parts''. As mentioned in Remark 1 following Theorem \ref{RankCrankEven}, this has an interesting interpretation in physics.

\subsection{Quantum modular forms}\label{QuantumModularForms}

In this section, we recall some definitions and examples of quantum modular forms and describe the quantum sets in Theorem \ref{QuantumWt32}. We begin with a few definitions (see \cite{KanekoZagier} for  additional background on quasimodular forms).

\begin{defn}
A function $f\colon\mathbb H\rightarrow\C$ is an {\bf almost holomorphic modular form of weight $k$} on a congruence subgroup $\Gamma$ if it transforms as a modular form of weight $k$ for $\Gamma$ and is a polynomial in $\frac 1{\operatorname{Im}(\tau)}$ with coefficients which are holomorphic on $\mathbb H\cup\mathbb P_1(\Q)$. Moreover, $f$ is a {\bf quasimodular form of weight $k$} if it is the constant term of an almost holomorphic modular form of weight $k$.
\end{defn}

\noindent Quantum modular are then defined as follows (see \cite{ZagierQuantum} for background on quantum modular forms).

\begin{defn}
For any infinite ``quantum set'' $\mathcal{Q}\subseteq\Q$, we say a function $f\colon \mathcal{Q}\rightarrow\C$ is a {\bf quantum modular form} of weight $k$ on a congruence subgroup $\Gamma$  if for all $\gamma\in\Gamma$, the coycle
\[r_{\gamma}(\tau):=f|_k(1-\gamma)(\tau)\]

\noindent extends to an open subset of $\R$ and is analytically ``nice''. Here ``nice'' could mean continuous, smooth, real-analytic etc. We say that $f$ is a \textbf{strong quantum modular form} if there is a formal power series over $\C$ attached to each point in $\mathcal{Q}$ with a stronger modularity requirement (see \cite{ZagierQuantum}). 

\end{defn}

\begin{remark} 

All of the quantum modular forms occurring in this paper have cocycles defined on $\R$ which are real-analytic except at one point. Moreover, they have full asymptotic expansions towards rational points in their quantum sets which agree with the asymptotic expansions of mock modular forms defined on the lower half-plane. 

\end{remark}

Especially relevant for us are certain partial theta functions which were shown to be quantum modular forms in recent work of Folsom, Ono, and Rhoades \cite{FOR}; namely 
\begin{equation*}
G(a,b;\tau):=\sum_{n\geq0}(-1)^nq^{\left(n+\frac ab\right)^2}.
\end{equation*}
For any $a, b\in\Z$ with $(a,b)=1$, $a>0$, define the following quantum set, where all fractions are assumed to have coprime denominator and numerator throughout 
\[\mathcal{Q}_{a,b}:=\left\{\frac hk\in\Q\colon h>0,\ b|2h,\ b\nmid h,\ k\equiv a\pmod b,\ k\geq a\right\}.\]
\noindent Since for $r=\frac{j}2\in\frac12+\Z$ 
\[\Theta_{\frac12}(N,r;\tau)=G\left(j,2N;\frac{N\tau}2\right),\] it suffices to study the quantum modular properties of $G(a,b;\tau)$. Although $a=0$ is excluded, it is easy to handle this case directly. Note that $G\left(0,1;\frac{\tau}2\right)$ is essentially a modular form as $G(0,1;\tau)=\frac{\eta(\tau)^2}{2\eta(2\tau)}+\frac12$ and also that $G(0,2N;\tau)=G(0,1;2N\tau)$. It is clear that $G\left(0,1;\frac{\tau}2\right)$ is quantum modular at any cusps where the eta quotient vanishes, namely for $\tau\in\left\{\frac hk\in\Q\colon k\equiv 1\pmod2\right\}$. For $a>0$, the situation is more subtle. Folsom, Ono and Rhoades proved that $G(a,b;\tau)$ have the following quantum properties: 
\begin{theorem}[\cite{FOR}]\label{FOR}
For $b$ even, $G(a,b;\tau)$ is a strong quantum modular form of weight $1/2$ with quantum set $\mathcal Q_{a,b}$. 

\end{theorem}
\begin{remark}

Although \cite{FOR} only states the theorem for $0<a<b$, an inspection of the proof shows that it is true for general integers $(a,b)=1$ with $a>0$ and $b$ even 
\end{remark}
When $N$ is even, we also have the analogous weight $\frac32$ partial theta functions $\Theta_{\frac32}(N,r;\tau)$ (see Theorem \ref{QuantumWt32}).

\section{The Fourier coefficients and partial theta functions of $A_N$}

In this section, we prove Theorem \ref{FourierCoeffsThomas}.

\begin{proof}[Proof of Theorem \ref{FourierCoeffsThomas}]
Define a subdomain of $Y$
\begin{equation}\nonumber
X:=\left\{ h\in Y| \text{Re}(C(h))>\text{Re}(\alpha(h))>0 \ \forall \,\alpha \in \Delta^+\right\}
\end{equation}
so that in particular $|\qq(x)|<|e^{(\delta-\varepsilon_{N+1})(x)}|<1$ for all $x$ in $X$. 
We begin with the following crucial lemma.  
\begin{lemma}\label{lem:expden}
As a function in $X$, we have
\begin{equation*}
\begin{split}
e^{\widehat{\rho}} \widehat{R} &=
e^{Nd-\rho_1}q^{-\frac{(N+1)(N+2)}{24}} \sum_{r\in \Z} \mathrm{sign}(-r) (-1)^{r} 
\qq^{-\frac{r^2}{2(N+1)}}\qq^{\frac{r}{2}} e^{r\frac{2\rho_1}{N+1}}\sum_{t\in T_r} \sum_{w\in W^\sharp} \sigma(w) w\left( e^t \qq^{\frac{t^2}{2N}}\right).
\end{split}
\end{equation*}
\end{lemma}
\begin{proof}
Inserting the statement of  Lemma \ref{lem:findenom} into the one of Lemma \ref{lem:affdenom} gives
\[ 
e^{\widehat{\rho}} \widehat{R} = \sum_{\alpha\in A_N}t_\alpha\left(e^{\widehat{\rho}}R\right)
 = \sum_{\alpha\in A_N}t_\alpha\left(\sum_{w\in W^\sharp}\sigma(w) w\left( \frac{e^{\widehat{\rho}}}{(1+e^{\delta-\varepsilon_{N+1}})}\right)\right).
\]
Using \eqref{eq:conj} and the bijectivity of the map $w\colon A_N\rightarrow A_N$ for every $w\in W^\sharp$, we get
\[
e^{\widehat{\rho}} \widehat{R} = \sum_{\alpha\in A_N}\sum_{w\in W^\sharp}\sigma(w)  w\left( t_{w^{-1}(\alpha)} \left(\frac{e^{\widehat{\rho}}}{(1+e^{\delta-\varepsilon_{N+1}})}\right)\right)
= \sum_{\alpha\in A_N}\sum_{w\in W^\sharp}\sigma(w)  w\left( t_{\alpha} \left(\frac{e^{\widehat{\rho}}}{(1+e^{\delta-\varepsilon_{N+1}})}\right)\right).
\]
Let $\alpha=m_1\alpha_1 +\dots +m_N \alpha_N$ be an element of $A_N$, and set $m_0:=m_{N+1}:=0$. By \eqref{eq:transl}, we have
\[
t_\alpha(\widehat\rho) = \widehat{\rho} +N\alpha -\left(\sum_{n=1}^{N+1}\left(m_n+\frac{N}{2}(m_n-m_{n-1})^2\right)\right)C\quad\text{and}
\quad
t_\alpha\left(\delta-\varepsilon_{N+1}\right)=\delta-\varepsilon_{N+1} -m_N C.
\]
Hence 
\[
e^{\widehat{\rho}} \widehat{R} =
 \sum_{m\in \Z^N}\sum_{w\in W^\sharp}\sigma(w)  w\left( e^{\widehat{\rho}} \prod\limits_{n=1}^{N+1}e^{Nm_n\alpha_n}\qq^{m_n+\frac{N}{2}(m_n-m_{n-1})^2} \left(1+e^{\delta-\varepsilon_{N+1}}\qq^{m_N}  \right)^{-1}\right),
\]
where we used the short-hand notation $m=(m_1,\dots,m_N)$ and kept as before $m_0=m_{N+1}=0$. Recall that $\alpha_n=\varepsilon_n-\varepsilon_{n+1}$. We split the exponential of the affine Weyl vector as
\[
e^{\widehat{\rho}}=e^{Nd-\rho_1}\prod_{n=1}^{N+1}e^{\frac{(N+2-2n)}{2}\varepsilon_n}.
\]
Note that $Nd-\rho_1$ is invariant under $W^\sharp$. 
Letting $q_n:=m_n-m_{n-1}$, we then find the identity
\[
\sum_{n=1}^{N+1} \left( \frac{N}{2}q_n^2+m_n\right)=
\frac{1}{2N}\sum_{n=1}^{N+1}\left( N\left(q_n+\frac{1}{2}\right)+(1-n)\right)^2 - \frac{(N+1)(N+2)}{24}.
\]
Defining the set
\[
S:=\left\{ (s_1,...,s_{N+1})\in \frac{1}{2}\Z^{N+1} \, \Big|\, s_n=N\left(q_n+\frac{1}{2}\right)+(1-n), q_n \in\Z, \sum_{n=1}^{N+1}q_n=0\right\},
\]
we obtain
\[
e^{\widehat{\rho}} \widehat{R} =
e^{Nd-\rho_1}\qq^{-\frac{(N+1)(N+2)}{24}}
 \sum_{s\in S}\sum_{w\in W^\sharp}\sigma(w)  w\left( \prod\limits_{n=1}^{N+1}e^{s_n\varepsilon_n}\qq^{\frac{s_n^2}{2N}}
\left(1+e^{\delta-\varepsilon_{N+1}}\qq^{-\frac{s_{N+1}}{N}-\frac{1}{2}}  \right)^{-1}\right).
\]

In the domain $X$ we can expand in a geometric series to find that $e^{\widehat{\rho}} \widehat{R}$ equals
\begin{equation*}
\begin{split}
e^{Nd-\rho_1}\qq^{-\frac{(N+1)(N+2)}{24}} \left(
\sum_{\substack {s\in S \\ s_{N+1}\leq -\frac{1}{2} }} \sum_{r=0}^\infty (-1)^r \sum_{w\in W^\sharp}\sigma(w)  w
\Bigg( e^{r(\delta-\varepsilon_{N+1})}\qq^{-r\left(\frac{s_{N+1}}{N}+\frac{1}{2}\right)}
\prod\limits_{n=1}^{N+1}e^{s_n\varepsilon_n}\qq^{\frac{s_n^2}{2N}} \Bigg) \right.\\ \left.- \sum_{\substack {s\in S \\ s_{N+1}> -\frac{1}{2}}} \sum_{r=1}^\infty (-1)^r \sum_{w\in W^\sharp}\sigma(w)  w
\left(  e^{-r\left(\delta-\varepsilon_{N+1}\right)}\qq^{r\left(\frac{s_{N+1}}{N}+\frac{1}{2}\right)}  \prod\limits_{n=1}^{N+1}e^{s_n\varepsilon_n}\qq^{\frac{s_n^2}{2N}} \right) \right).
\end{split}
\end{equation*}
Since the double sum converges absolutely in the domain $X$, we can interchange summations.
Define
\[
 g_r:= \begin{cases}
   \sum\limits_{\substack {s\in S \\ s_{N+1}\leq -\frac{1}{2} }}\sum\limits_{w\in W^\sharp}\sigma(w)  w
\left(  e^{-r\left(\delta-\varepsilon_{N+1}\right)}\qq^{r\left(\frac{s_{N+1}}{N}+\frac{1}{2}\right)}  \prod\limits_{n=1}^{N+1}e^{s_n\varepsilon_n}\qq^{\frac{s_n^2}{2N}} \right) & \quad \text{if} \ r\leq 0,\\     
\sum\limits_{\substack {s\in S \\ s_{N+1}> -\frac{1}{2} }}\sum\limits_{w\in W^\sharp}\sigma(w)  w
\left(  e^{-r\left(\delta-\varepsilon_{N+1}\right)}\qq^{r\left(\frac{s_{N+1}}{N}+\frac{1}{2}\right)}  \prod\limits_{n=1}^{N+1}e^{s_n\varepsilon_n}\qq^{\frac{s_n^2}{2N}} \right) & \quad \text{if} \ r> 0.     
  \end{cases}
\]
Then
\[
e^{\widehat{\rho}} \widehat{R}=
e^{Nd-\rho_1}\qq^{-\frac{(N+1)(N+2)}{24}} \sum_{r\in\Z}(-1)^r \text{sign}(-r) g_r.
\]
We can express $\varepsilon_{N+1}-\delta$ in terms of the odd Weyl vector and positive even simple roots:
\begin{equation}\label{eq:oddroots}
\begin{split}
\varepsilon_{N+1}-\delta &= -\delta +\frac{1}{N+1}\left(\sum_{n=1}^{N+1}\varepsilon_n +\sum_{n=1}^N n(\varepsilon_{n+1}-\varepsilon_n)\right) 
= \frac{2}{N+1}\rho_1 +\varepsilon_{N+1}-\frac{1}{N+1}\sum_{n=1}^{N+1}\varepsilon_n.
\end{split}
\end{equation}
We see that $\varepsilon_{N+1}-\delta-\frac{2}{N+1}\rho_1$ is in $\frac{1}{N+1}A_N$. 
For $(s_1,...,s_{N+1})\in S$, we find that 
\begin{equation}\label{foo}
\left(s_{N+1}-\frac{r}{N+1}+r\right)^2+\sum_{n=1}^N\left(s_n-\frac{r}{N+1}\right)^2=
\frac{Nr^2}{N+1}+2rs_{N+1}+\sum_{n=1}^{N+1} s_n^2.
\end{equation}
Combining \eqref{eq:oddroots} and (\ref{foo}), we can rewrite 
\[
 e^{-r\left(\delta-\varepsilon_{N+1}\right)}\qq^{r\left(\frac{s_{N+1}}{N}+\frac{1}{2}\right)}  \prod\limits_{n=1}^{N+1}e^{s_n\varepsilon_n}\qq^{\frac{s_n^2}{2N}}=
e^{r\frac{2\rho_1}{N+1}} \qq^{-\frac{r^2}{2(N+1)}}\qq^{\frac{r}{2}}\prod_{n=1}^{N+1}e^{t_n\varepsilon_n}\qq^{\frac{t_n^2}{2N}}
\]
with $t_n:=s_n-\frac{r}{N+1}+r\delta_{n,N+1}$.
Then 
\begin{equation*}
\begin{split}
\sum_{n=1}^{N+1}t_n\varepsilon_n &= \sum_{n=1}^{N+1}\sum_{j=1}^n t_j (\varepsilon_j-\varepsilon_{j+1})
=  \sum_{n=1}^{N}\sum_{j=1}^n t_j\alpha_n.
\end{split}
\end{equation*}
Here we used that $t_1+\dots+t_{N+1}=0$, which follows from the same property for the $s_n$. 
Let $q_j$ be as in the definition of the set $S$; in particular we can write $q_j=m_j-m_{j-1}$ with
integers $m_j$ for $1\leq j \leq N$, and $m_{N+1}=0$. Then
\begin{equation*}
\begin{split}
 \sum_{j=1}^n t_j &= -\frac{rn}{N+1}+\sum_{j=1}^n s_j 
 = \frac{n(N-n+1)}{2}-\frac{rn}{N+1}+\sum_{j=1}^n Nq_j
=\frac{(N-n+1)n}{2}-\frac{rn}{N+1}+Nm_n.
 \end{split}
\end{equation*}
Using the sets $T_r$ \eqref{eq:setTr}, we finally get
\[
g_r = e^{r\frac{2\rho_1}{N+1}} \qq^{-\frac{r^2}{2(N+1)}}\qq^{\frac{r}{2}} \sum_{t\in T_r} \sum_{w\in W^\sharp} \sigma(w) w\left( e^t \qq^{\frac{t^2}{2N}}\right).
\]
\end{proof}

Letting $\zz=-e^{\frac{2\rho_1}{N+1}}$, we deduce the following.
\begin{corollary}
The identity $A=BC$ holds as functions on $X$, where 
\begin{equation*}
\begin{split}
A &:= \prod_{j=1}^{N+1} \frac{\qq^{\frac{1}{24}}}{(-e^{\delta-\varepsilon_j};\qq)_\infty(-e^{\varepsilon_j-\delta}\qq;\qq)_\infty},\qquad
B := \frac{\qq^{-\frac{(N+1)^2}{24}}}{(\qq;\qq)_\infty^{N+1}\prod\limits_{\alpha\in \Delta_0}(e^\alpha\qq;\qq)_\infty(e^{-\alpha}\qq;\qq)_\infty},\\
C&:=  \sum_{r\in\Z}\mathrm{sign}(-r) \zz^{r}\qq^{-\frac{r^2}{2(N+1)}}\qq^{\frac{r}{2}} 
\sum_{t\in T_r}   v_t \qq^{\frac{t^2}{2N}}. 
\end{split}
\end{equation*}

\end{corollary}
\begin{proof}
The corollary follows immediately from Lemma \ref{lem:expden} by inserting $v_t$ in the definition of $C$ in \eqref{eq:vmu}. 
\end{proof}
Evaluating the expressions in this equality provides a nice expansion of $\phi_{N+1}(z;\tau)$.

\begin{corollary}\label{cor:coeff}
Inside the range $|q| < |\zeta| < 1$, we have
\begin{equation*}
\begin{split}
\phi_{N}(z;\tau) &= \frac{i^{N}}{d_{N-1} \eta(\tau)^{N^2+N}}  \sum_{r\in\Z} \zeta^{r +\frac{N}{2}}
\mathrm{sign}(r) q^{-\frac{\left(r+\frac{N}{2}\right)^2}{2N}}    \sum_{t\in T_r } 
\prod_{\alpha\in{\Delta_0^+}}(t|\alpha) q^{\frac{t^2}{2(N-1)}} .
\end{split}
\end{equation*}
\end{corollary}
\begin{proof}
The evaluation $v$ maps every regular exponential $e^\lambda$ for $\lambda\in \frac{1}{N}A_{N-1}$ to 1.
The application of $v$ to $A$ and $B$ is finite for  $|\qq(x)| < |\zz^{-1}(x)| < 1$ and $x\in X$,  
and the same is true for $C$ by Lemma \ref{lemma:vvmu}. 
The identity \eqref{eq:oddroots} implies that $v\left(e^{\delta-\varepsilon_j}\right)=e^{-\frac{2\rho_1}{N}}=-\zeta$ for all $j=1,\dots,N$, so that 
\[
 v(A) = \qq^{\frac{N}{24}}\left(\zz^{-1};\qq\right)_{\infty}^{-N}(\zz\qq;\qq)_{\infty}^{-N}\quad\text{and}\quad
v(B) = \qq^{-\frac{N^2}{24}}(\qq;\qq)_\infty^{-N^2}.
\]
By Lemma \ref{lemma:vvmu},
\[ 
v(v_t) = \frac{1}{d_{N-1}}\prod_{\alpha\in{\Delta_0^+}}(t|\alpha)
\]
and the evaluation $v(C)$ follows. All three evaluations 
$v(A), v(B), v(C)$ are meromorphic functions on 
$\left\{ x=-2\pi i \tau d +\frac{4\pi  i z h_{\rho_1}}{N-1}\colon \text{Im}(\tau)> \text{Im}(z)>0 \right\}$, so that  
the result follows with $\zeta=\zz^{-1}(x)$ and $q=\qq(x)$. 
\end{proof}
This completes the proof as Corollary \ref{cor:coeff} and Lemma \ref{lemma:vvmu} imply Theorem \ref{FourierCoeffsThomas}.
\end{proof}
The case $N=1$ can be proven in a very similar manner using \eqref{CrankPartialFrac}, which is the denominator identity of $\widehat{\mathrm{g}\ell}(1|1)$ (see Example 4.1 of \cite{KW1}).
\begin{example}\label{prop:N=1}
The Fourier coefficients of $\phi_1(z;\tau)$ are given by 
\[
\chi(1,r;\tau) = \frac{iq^{-\frac{r^2}{2}}}{\eta(\tau)^3}\sum_{m=0}^\infty (-1)^m q^{\frac{\left(m+\left| r-\frac{1}{2}\right|+\frac{1}{2}\right)^2}{2}}.
\]
\end{example}

\begin{proof}
Suppose $|q|<|\zeta|<1$. Expanding \eqref{CrankPartialFrac} in a geometric series and rewriting easily gives the statement.
\end{proof}
\section{Second viewpoint on the decomposition into partial theta functions}\label{SecondViewpoint}

In this section, we prove Theorem \ref{RankCrankEven} and use it to extract the Fourier coefficients of $\phi_{M,N}$ in Theorem 1.4.
A key ingredient for the proof of Theorem 1.3 is the following result whose proof is deferred to Section \ref{ProofLemmaQuasiEllipticCancel}.

\begin{lemma}\label{QuasiEllipticCancel}
For $N\in\N$, there exist meromorphic functions $f^*_j(\tau)$ for $0\leq j\leq\frac{N-1-\delta_e}2$ with $f^*_{\frac{N-1-\delta_e}2}(\tau)\neq0$ such that for all $r\in\Z$
\[\sum_{j=0}^{\frac{N-1-\delta_e}2}f^*_j(\tau)\mathcal D_{q}^{j}\left(\vartheta_{\frac12+\delta_e}(N,r;\tau)\right)=0.\]

\end{lemma}

\subsection{Proof of Theorem \ref{RankCrankEven} for $M=0$}

The first step in the proof of Theorem 1.3 is to show the following decomposition for the case when $M=0$:

\begin{proposition}\label{EvenRankCrankFirst}
For $N\in\N$ there exist meromorphic functions $g_j(\tau)$ such that 
\[\phi_N(z;\tau)=\sum_{j=0}^{\frac{N-1-\delta_e}2}g_{j}(\tau)\mathcal D_w^{2j+\delta_e}\left(F_N(z,u;\tau)\right)\big|_{w=1}.\]

\end{proposition}

\begin{proof}

We first determine the elliptic transformations of $F_N(z;\tau)$ and prove that, although this function does not in general transform as a negative index Jacobi form, we can ``correct'' the elliptic transformations to match those of $\phi_N(z;\tau)$. The following periodicity property is evident:

\[F_N(z+1,u;\tau)=(-1)^{N}F_N(z,u;\tau).\]
For the elliptic transformation $z\mapsto z+\tau$, a direct calculation gives
\[ (-1)^N \zeta^{-N}q^{-\frac N2}F_N(z+\tau,u;\tau)-F_N(z,u;\tau)\]\[=\sum_{r=0}^{N-1}\zeta^{r-\frac N2}q^{-\frac{1}{2N}\left(r-\frac N2\right)^2}\sum_{n\in\Z}(-1)^{Nn}w^{Nn+r-\frac N2}q^{\frac N2 \left(n-\frac12+\frac rN\right)^2}.\]

\noindent Thus, we have the following elliptic transformation formula for the iterated derivative of $F_N(z;\tau)$:

\[ (-1)^N \zeta^{-N}q^{-\frac N2}\mathcal D_w^{2j+\delta_e}(F_N(z+\tau,u;\tau))\big|_{w=1}-\mathcal D_w^{2j+\delta_e}(F_N(z,u;\tau))\big|_{w=1}\]
\[=2^jN^{j+\delta_e}\sum_{r=0}^{N-1}\zeta^{r-\frac N2}q^{-\frac{1}{2N}\left(r-\frac N2\right)^2}\mathcal D_{q}^j\left(\vartheta_{\frac12+\delta_e}(N,r;\tau)\right).\]

We now use the functions $f_j^*$ from Lemma 5.1 to ``correct'' the elliptic transformation by defining

 \[P_N(z;\tau):=\sum_{j=0}^{\frac{N-1-\delta_e}2}\frac{f^*_{j}(\tau)}{2^jN^{j+\delta_{\text {e}}}}\mathcal D_{w}^{2j+\delta_e}(F_N(z,u;\tau))\big|_{w=1},\]
so that
 \begin{equation*}
 P_N(z+1;\tau)=\zeta^{-N}q^{-\frac N2}P_N(z+\tau;\tau)=(-1)^NP_N(z;\tau).
 \end{equation*}

\noindent
Thus, $P_N(z;\tau)$ satisfies the same elliptic transformations as $\phi_N(z;\tau)$. 
It also has poles in the same locations and of the same order; namely poles in $\Z\tau+\Z$ of order $N$.
Hence the product \[p_N(z;\tau):=\vartheta(z;\tau)^NP_N(z;\tau)\] is an entire elliptic function and therefore constant in $z$. It remains to show that $P_N(z;\tau)\neq0$, which we prove by looking at the behavior as $z\rightarrow0$. The principal part as $z\rightarrow0$ of $\mathcal D_w^{j}\left(F_N(z,u;\tau)\right)$ only comes from the $n=0$ term in (\ref{defineFN}), which contributes 
\[\frac{(\zeta w)^{\frac N2}}{1-\zeta w}=-\sum_{m\geq0}\frac{B_m\left(\frac N2\right)(2\pi i (u+z))^{m-1}}{m!}=-\frac1{2\pi i(u+z)}+O(1),\]
where $B_m(x)$ is the usual $m$-th Bernoulli polynomial. Thus, as $z\to 0$,

\begin{equation}\label{PrincPart}\mathcal D_w^j(F_N(z,u;\tau))\big|_{w=1}=\frac{(-1)^{j+1}j!}{(2\pi i z)^{j+1}}+O(1)\end{equation}

\noindent and so

\[P_N(z;\tau)=\frac{(-1)^{N}(N-1)!f_{\frac{N-1-\delta_e}2}^*(\tau)}{N^{\delta_e}(2N)^{\frac{N-1-\delta_e}2}(2\pi i z)^N}+O\left(z^{-N+1}\right).\]
We can then use the well-known formula
\begin{equation*}\vartheta'(0;\tau)=-2\pi\eta(\tau)^3,\end{equation*}
and compare the coefficients of $z^{-N}$ to give

\[p_N(z;\tau)=\frac{(N-1)!(-i)^{N}f_{\frac{N-1-\delta_e}2}^*(\tau)}{N^{\delta_e}(2N)^{\frac{N-1-\delta_e}2}}\neq0,\]
as by assumption $f_{\frac{N-1-\delta_e}2}^*\neq0$. By absorbing the constants into the $f^*_j$, Proposition \ref{EvenRankCrankFirst} follows. 
\end{proof}
\begin{proof}[Proof of Theorem \ref{RankCrankEven} for $M=0$]

To finish the proof for $M=0$, we connect the functions $g_j$ in Proposition \ref{EvenRankCrankFirst} to the Laurent coefficients of $\phi_N$ given in \eqref{LaurentCoeffsDefn} by comparing principal parts. Namely, using (\ref{PrincPart}), we easily read off:
 \[g_{j}(\tau)=\frac{(-1)^{\delta_e+1}D_{2j+\delta_e+1}(\tau)}{\left(2j+\delta_e\right)!}.
 \]

\end{proof}

\subsection{Proof of Lemma \ref{QuasiEllipticCancel}}\label{4.2}\label{ProofLemmaQuasiEllipticCancel}

For $N$ odd, Lemma 2.1 of \cite{ZwegersRankCrankPDE} easily gives Lemma \ref{QuasiEllipticCancel} by rearranging terms. The condition $f_0\neq0$ (in the notation of \cite{ZwegersRankCrankPDE}) is not stated explicitly in the statement; however the proof shows that one can choose $f_0=1$ in Lemma 2.1 of \cite{ZwegersRankCrankPDE}. Now suppose that $N$ is even. For $k\in\N$, consider the Ramanujan-Serre derivative, which raises the weight of a modular form by $2$:
\begin{equation*}
\mathcal E_k:=\mathcal D_{q}-\frac{k}{12}E_2(\tau),
\end{equation*} and its iterated version starting at weight $\frac32$ given by $\mathcal{E}^n:=\mathcal E_{2n-\frac12}\circ \mathcal E_{2n-\frac52}\circ \ldots\circ \mathcal E_{\frac72}\circ\mathcal E_{\frac32}$.
By rearranging, it is enough to show that there are holomorphic functions $f_j$ such that for all $r\in\Z$
 
\[\sum_{j=0}^{\frac N2-1}f_{j}(\tau)\mathcal{E}^{j}\left(\vartheta_{\frac32}(N,r;\tau)\right)=0.\]

This is clearly equivalent to the following, where $\widetilde{\vartheta}_{\frac12+\nu}(N,r;\tau)$ is defined in (\ref{ThetaTilde}).

\begin{lemma}
If $N\in 2\N$, then there exist meromorphic functions $f_j(\tau)$ with $f_{\frac N2-1}(\tau)\neq0$ such that for all $r\in\Z$
\begin{equation}\label{QuasiZeroSum}\sum_{j=0}^{\frac N2-1}f_{j}(\tau)\mathcal{E}^{j}\left(\widetilde{\vartheta}_{\frac32}(N,r;\tau)\right)=0.\end{equation}
\end{lemma}
\begin{proof} 
The approach taken here is similar to Zwegers' proof of Lemma 2.1 in \cite{ZwegersRankCrankPDE}, although we give details for the reader's convenience. Using (\ref{ThetaShiftN}) and (\ref{ThetaNegateN}), it suffices to prove the lemma for $1\leq r\leq \frac N2-1$. By (\ref{ThetaSpecial}), we may simply choose $f_0=1$ for $N=2$.
Thus, we assume for the remainder of the proof that $N\geq4$. 

Consider the vector-valued form
 \[\overset{\rightarrow}{\vartheta}_N(\tau):=\left(\widetilde{\vartheta}_{\frac32}(N,1;\tau),\ldots,\widetilde{\vartheta}_{\frac32}\left(N,\frac N2-1;\tau\right)\right)^T\] 
 and the matrix-valued form 
 \[T_N:=\left(\overset{\rightarrow}{\vartheta}_N,\mathcal{E}\left(\overset{\rightarrow}{\vartheta}_N\right),\ldots,\mathcal{E}^{\frac N2-2}\left(\overset{\rightarrow}{\vartheta}_N\right)\right).\]
Using Proposition \ref{ThetaVectorTrans}, we see that
 \[T_N(\tau+1)=\operatorname{diag}\left(e\left(\frac{j^2}{2N}\right)\right)_{1\leq j\leq\frac N2-1}T_N(\tau),\]
 \[T_N\left(-\frac1{\tau}\right)=\frac2{\sqrt N}(-i\tau)^{\frac32}\left(\sin\left(\frac{2\pi k\ell}N\right)\right)_{1\leq k,\ell\leq\frac N2-1}T_N(\tau)\operatorname{diag}\left(\tau^{2j-2}\right)_{1\leq j\leq\frac N2-1}.\]
Hence
 \[\det\left(T_N\right)(\tau+1)=e\left(\frac{(N-1)(N-2)}{48}\right)\det\left(T_N\right)(\tau),\]
 \[\det\left(T_N\right)\left(-\frac1{\tau}\right)=(-i\tau)^{\frac{(N-1)(N-2)}{4}}\det\left(T_N\right)(\tau).\]
Here we used the following elementary determinant formula:
 \[\det\left(\frac{2(-i)^{\frac32}}{\sqrt N}\sin\left(\frac{2\pi \ell k}{N}\right)\right)_{1\leq k,\ell\leq\frac N2-1}=(-i)^{\frac{(N-1)(N-2)}4}.\]
 From these transformations we see that $\det\left(T_N\right)$ is a modular form on $\operatorname{SL}_2(\Z)$ with the same multiplier system as $\eta^{\frac{(N-1)(N-2)}{2}}$. It is then easy to see that the ratio \[\frac{\det\left(T_N\right)(\tau)}{\eta(\tau)^{\frac{(N-1)(N-2)}{2}}}=:\alpha\] is a constant.  Assuming that $\alpha$ is non-zero, it follows that $T_N(\tau)$ is invertible for all $\tau\in\H$.

To see that $\alpha$ is non-zero, observe that by elementary row operations we can write

\[\operatorname{det}(T_N)=\operatorname{det}\left(\overset{\rightarrow}{\vartheta}_N,\mathcal D_{q}\left(\overset{\rightarrow}{\vartheta}_N\right),\ldots,\mathcal D_{q}^{\frac N2-2}\left(\overset{\rightarrow}{\vartheta}_N\right)\right).\]
For $1\leq r\leq \frac N2-1$, note that $\widetilde{\vartheta}_{\frac32}(N,j;\tau)=\frac jN q^{\frac{j^{2}}{2 \, N}}\left(1+O(q)\right)$. But then we have
\begin{equation}\label{TNDet}\det(T_N)(\tau)=\det\left(\operatorname{diag}\left(\frac jNq^{\frac{j^2}{2N}}\right)_{1\leq j\leq\frac N2-1}\left(\left(\frac{m^2}{2N}\right)^{k-1}+O(q)\right)_{1\leq m,k\leq\frac N2-1}\right).\end{equation}
The right hand side is clearly non-zero as the first matrix on the right hand side of (\ref{TNDet}) is diagonal with non-zero entries and the second is a Vandermonde matrix with distinct generators. 
We can now solve (\ref{QuasiZeroSum}) by writing it as 
\[T_N\left(f_0,\ldots,f_{\frac N2-2}\right)^T+f_{\frac N2-1}\mathcal E^{\frac N2-1}\left(\overset{\rightarrow}{\vartheta}_N\right)=0.\]
Choosing $f_{\frac N2-1}=1$, it is easy to see that we obtain holomorphic functions $f_j$ satisfying (\ref{QuasiZeroSum}) as follows:
\[\left(f_0,\ldots,f_{\frac N2-2}\right)^T=-T_N^{-1}f_{\frac N2-1}\mathcal E^{\frac N2-1}\left(\overset{\rightarrow}{\vartheta}_N\right).\]

\end{proof}

\subsection{Proof of Theorem 1.3}
We now use the above decomposition of $\phi_N(z;\tau)$ to deduce Theorem 1.3 for any $M$ satisfying the conditions of the theorem. 
We may define a \emph{heat operator}, $\mathcal H:=\mathcal H_{-\frac N2}:=2N\mathcal D_q+\mathcal D_{\zeta}^2$ which has the property that it preserves the elliptic transformations of a function satisfying the elliptic transformation properties of an index $-\frac N2$ Jacobi form (note that this differs by a constant from the heat operator defined in \cite{EichlerZagier}). The main idea is to show the following:

\begin{proposition}\label{HeatOperatorCancel}
For any $N\in\N$ and $M\in\N_0$, there exist meromorphic functions $f_j(\tau)$ for $j=0,1,\ldots, M$ with $f_M(\tau)\neq0$ such that
\begin{equation}\label{ToShow}\phi_{2M,N+2M}(z;\tau)=\sum_{j=0}^Mf_j(\tau)\mathcal H^j\left(\phi_N(z;\tau)\right).\end{equation}
\end{proposition}

For the proof we need a few preliminaries on special determinants and orthogonal polynomials.
We recall that a matrix $H$ is called a \emph{Hankel matrix} if it is constant on each skew diagonal. It is well known that the determinants of Hankel matrices are connected to 
orthogonal polynomials and continued fractions. Such functions have a long history which is explained in many places; we refer the reader to Chapter 11 of \cite{Wall} for more details. Given a sequence of numbers $c_0,c_1,c_2,\ldots$ which are the moments of a sequence of orthogonal polynomials $p_n(x)$, we say that $p_n(x)$ is a sequence of orthogonal
polynomials \emph{relative to $c_n$}. For any sequence $\left(c_n\right)_{n\in\N_0}$, we define the following sequence of Hankel determinants:
\[\Delta_n:=\left\vert\begin{matrix}c_0&c_1&\cdots&c_n\\ c_1&c_2&\cdots &c_{n+1}\\ \vdots&\vdots&\vdots&\vdots\\ c_n&c_{n+1}&\cdots&c_{2n}\end{matrix}\right\vert.\]
Then we have the following:
\begin{theorem}[\cite{Wall} Theorem 50.1]\label{HankelDetTheorem}
Let $c_n$ be a sequence of numbers. Then $\Delta_n\neq0$ for all $n$ if and only if there exists a sequence of 
orthogonal polynomials relative to $c_n$.
\end{theorem}

We now proceed to the proof of Proposition \ref{HeatOperatorCancel}. 

\begin{proof}[Proof of Proposition \ref{HeatOperatorCancel}]
First note that $\phi_{N}$ and $\phi_{2M,N+2M}$ have the same elliptic transformation properties. Moreover, the right hand side of (\ref{ToShow}) has poles of order exactly $N+2M$ for $z\in\Z+\Z\tau$, as does $\phi_{2M,N+2M}$. It suffices to choose $f_j$ such that the right hand side has zeros of order at least $2M$ for $z\in\frac12+\Z+\Z\tau$, as dividing the right hand side of (\ref{ToShow}) by the left hand side gives a holomorphic elliptic function, and hence a constant.

Thus, it suffices to choose $f_j$ to cancel out the first $2M$ Taylor coefficients of (\ref{ToShow}) at $z=-\frac12$. Note, that as $\vartheta\left(z+\frac12\right)$ is even, we only have even order Taylor coefficients to cancel out. We expand 
\[\mathcal H^j\left(\phi_N(z;\tau)\right)=:\sum_{k\geq0}T_{2k,j}(\tau)\left(z+\frac12\right)^{2k}.\]
Thus we need to find a solution of the equation
\[T\begin{pmatrix}f_0\\ \vdots\\ f_{M-1}\end{pmatrix}+f_M\begin{pmatrix}T_{0,M}\\ \vdots\\ T_{2M,M}\end{pmatrix}=0,\]
where
\[T:=\begin{pmatrix}T_{0,0}&\cdots&T_{0,M-1}\\ \vdots&\vdots&\vdots\\ T_{2M-2,0}&\cdots&T_{2M-2,M-1}\end{pmatrix}.\]

\noindent Choosing $f_M(\tau)=1$ it suffices to prove that $\det(T)(\tau)\not\equiv0$. To show this, it is enough to show that the $q$-series for $\det(T)$ has at least one non-vanishing coefficient; we look at the lowest order term. By the definition of $\vartheta(z;\tau)$, we have that 
\[\vartheta(z;\tau)=q^{\frac18}\left(e^{\pi i\left(z+\frac12\right)}+e^{-\pi i\left(z+\frac12\right)}\right)+O\left(q^{\frac98}\right).\]

By rescaling, it suffices to study the matrix of Taylor coefficients of $\Psi(v;\tau):=q^{-\frac N8}\sec^N(v)$ and its iterated heat operators. By row reduction, we can replace the action of the operator $\mathcal H^j$ with $\frac{\partial^{2j}}{\partial v^{2j}}$ and study the resulting determinant. By rescaling the $k$-th row of the resulting matrix by multiplying by $(2k)!$, we need to show nonvanishing of 
\[S:=\left\vert\left(\frac{\partial^{2(j+k)}}{\partial v^{2(j+k)}}\left(\sec^N(v)\right)\bigg|_{v=0}\right)_{0\leq j,k\leq M-1}\right\vert.\]

\noindent These numbers are known as the (signless) higher order Euler numbers, and we define
\[E_{2j}^{(N)}:=\frac{\partial^{2j}}{\partial v^{2j}}\left(\sec^N(v)\right)\bigg|_{v=0}.\]
By Theorem \ref{HankelDetTheorem}, it suffices to show that the higher order Euler numbers are for each $N$ a moment sequence for a sequence of orthogonal polynomials. But Lemma 1.3 of \cite{MeixnerRef} gives the desired sequence of orthogonal polynomials for any $N$. 

\end{proof}

Finally, we compare Laurent expansions to determine a more convenient decomposition. Using the following geometric series expansion:
\begin{equation}\label{FNExpansion}F_N(z,u;\tau)=\sum_{\substack{n\geq0\\ r\geq\frac N2}}(-1)^{Nn}q^{\frac N2n^2+nr}\zeta^rw^{r+Nn}-\sum_{\substack{n\geq1\\ r<\frac N2}}(-1)^{Nn}q^{\frac N2n^2-nr}\zeta^rw^{r-Nn},\end{equation} where $r$ runs through $\frac N2+\Z$, we directly find that \[\mathcal H^j_{-\frac N2}\left(F_N(z,u;\tau)\right)=\mathcal D_w^{2j}\left(F_N(z,u;\tau)\right).\]
Since we have already shown Theorem 1.3 for $M=0$ in Section 5.1, we see that for any $N\in \N$, $M\in\N_0$ that there exist meromorphic functions $g_j(\tau)$ such that
\[\phi_{2M,N+2M}(z;\tau)=\sum_{j=0}^{\frac{N-1-\delta_e}2+M}g_j(\tau)\mathcal D_w^{2j+\delta_e}\left(F_N(z,u;\tau)\right)\big|_{w=1}.\]
Using (\ref{PrincPart}) 
gives Theorem 1.3.

\subsection{Extracting the Fourier coefficients of $\phi_{2M,N+2M}$}

In this section, we use Theorem \ref{RankCrankEven} to deduce Theorem \ref{EvenDecomp} by expanding the Appell-Lerch series. Recall that throughout, we fix the range $0\leq\operatorname{Im}(z)<\operatorname{Im}(\tau)$. We also assume that $0<\operatorname{Im}(u)<\operatorname{Im}(\tau)-\operatorname{Im}(z)$.

\begin{proof}
We use the geometric series expansion given in (\ref{FNExpansion}) to pick off the $\zeta^r$ coefficient of $F_N(z,u;\tau)$ as 

\[\left[\zeta^r\right]\left(F_N(z,u;\tau)\right)=\begin{cases}  q^{-\frac{r^2}{2N}}\sum_{n\geq0}(-1)^{Nn}q^{\frac N2\left(n+\frac rN\right)^2}w^{N\left(n+\frac rN\right)}  &\text{ if } r\geq\frac N2,\\ -q^{-\frac{r^2}{2N}}\sum_{n\geq1}(-1)^{Nn}w^{-N\left(n-\frac rN\right)}q^{\frac N2\left(n-\frac rN\right)^2} & \text{ if } r<\frac N2.\end{cases}\]

Differentiating and applying Theorem \ref{RankCrankEven} completes the proof.

\end{proof}
\section{Proof of Theorem 1.5}\label{QuantumProofs}
 In this section, we complete the proof of Theorem 1.5 by showing strong quantum modularity of $\Theta_{\frac32}(N,r;\tau)$ for any $N\in2\N$, $r\in\Z$. We find it convenient to take sums and differences of two Fourier coefficients (which makes describing the asymptotic expansions in Proposition 6.4 much cleaner).
Specifically, we define
\[\Theta_{\frac32}^{\pm}(N,r;\tau):=\Theta_{\frac32}(N,r;\tau)\pm\Theta_{\frac32}(N,-r;\tau).\]
One easily sees that 
\begin{equation}\label{Theta32Difference}\Theta_{\frac32}^-(N,r;\tau)=\widetilde{\vartheta}_{\frac32}\left(N,r;\tau\right)+\frac rNq^{\frac{r^2}{2N}},\end{equation}
and for $0\leq r\leq N-1$ we compute:

\begin{equation}\label{Theta32Sum}\Theta_{\frac32}^+(N,r;\tau)=\frac 1N\sum_{\substack{n\geq-r\\ n\equiv\pm r\pmod N}}nq^{\frac{n^2}{2N}},\end{equation}
where for $r=0$, we consider each summand with multiplicity 2.

\subsection{Quantum sets}

In this section, we show that the following is a quantum set for $\Theta_{\frac32}(N,r;\tau)$:

\begin{equation}\label{QuantumSetDefn}\widehat{\mathcal Q}_{N,r}:=\begin{cases}\left\{\frac hk\in\Q\colon \frac N2|k,\ \operatorname{ord}_2(k)=\operatorname{ord}_2(N)-1\right\}&  \text{ if }  \frac{N}{2}\nmid r,\\ \left\{\frac hk\in\Q\colon \operatorname{ord}_2(k)>\operatorname{ord}_2(N)\right\}&\text{ if } r\equiv\frac N2\pmod N,\\ \left\{\frac hk\in\Q\colon \operatorname{ord}_2(k)=\operatorname{ord}_2(N)\right\}&\text{ if } r\equiv0\pmod N.\end{cases}
\end{equation} 
Here $\operatorname{ord}_2(n)$ is the $2$-adic valuation of an integer $n$. We will later see that the congruence subgroup for $\Theta_{\frac32}(N,r;\tau)$ is $\Gamma_1(2N)$, which acts on $\widehat{\mathcal Q}_{N,r}$. Note that it is enough to determine a quantum set for each $0\leq r\leq N-1$, as if we shift $r\mapsto r\pm N$ in $\Theta_{\frac32}(N,r;\tau)$, we only change the function by a finite polynomial in rational powers of $q$.

\begin{lemma}
For any $N\in2\N$, $0\leq r\leq N-1$, $\Theta_{\frac32}(N,r;\tau)$ is well-defined on $\widehat{\mathcal{Q}}_{N,r}$. 

\end{lemma} 
\begin{proof} 
Throughout, we abuse terminology to say that a partial theta function is \emph{convergent} at a root of unity if it has a $q$-hypergeometric representation which is. 

We start with $\frac N2\nmid r$. Recall the following beautiful identity of Warnaar (p. 17 of \cite{Warnaar})
\begin{equation}\label{WarnaarFormula}\sum_{n\geq0}\frac{\left(q;q^2\right)_n\left(Aq;q^2\right)_n(Aq)^n}{(-Aq)_{2n+1}}=\sum_{n\geq0}(-A)^nq^{n(n+1)}.\end{equation} Multiplying both sides of (\ref{WarnaarFormula}) by $A^{\frac rN}$, applying $\mathcal D_A$, and then setting $A=-q^{\frac{2r}N-1}$, we see that

\[\sum_{n\geq0}\left(n+\frac rN\right)q^{\left(n+\frac rN\right)^2}\] 
is convergent at any root of unity $q$ such that $(q;q^2)_{\infty}(-q^{\frac{2r}N};q^2)_{\infty}=0$ and $(q^{\frac{2r}N})_{\infty}\neq0$. Making the change of variables $q\rightarrow q^{\frac{N}2}$, shows that $\widehat{\mathcal{Q}}_{N,r}$ is a quantum set for $\Theta_{\frac32}(N,r;\tau)$. We note that other hypergeometric representations in the literature could also be used to derive the same quantum set; see for example (1.2) discussed by Alladi in \cite{Alladi}.

If $r=\frac N2$, note that $\Theta_{\frac32}\left(N,\frac N2;\tau\right)=\Theta_{\frac32}\left(1,\frac12;N\tau\right)$, so we first study $\Theta_{\frac32}\left(2,1;\tau\right)$. This function can be split into two pieces, where the first is \[\sum_{n\geq0}q^{\frac{\left(n+\frac12\right)^2}2}=\frac{\eta(2\tau)^2}{\eta(\tau)}.\]
The eta-quotient on the right hand side is easily seen to vanish at any cusp $\frac hk$ with $2|k$.  
The second piece is related to a ``sum of tails'' by
(3.11) of  \cite{AndrewsJimenezUrrozOno}, which states that 
\begin{equation}\label{RZeroSOT}\sum_{n\geq1}nq^{\frac{n^2+n}{2}}=\frac{\left(q^2;q^2\right)_{\infty}}{\left(q;q^2\right)_{\infty}}\sum_{n\geq1}\frac{(-1)^nq^n}{1-q^n}+\sum_{n\geq0}\left(\frac{\left(q^2;q^2\right)_{\infty}}{\left(q;q^2\right)_{\infty}}-\frac{\left(q^2;q^2\right)_{n}}{\left(q;q^2\right)_{n+1}}\right).\end{equation}

The right hand side of (\ref{RZeroSOT}) terminates for even order roots of unity. By combining, the series $\Theta_{\frac32}(2,1;\tau)$ converges at even order roots of unity, and rescaling $q\mapsto q^N$ gives the quantum set in the definition of $\widehat{\mathcal Q}_{N,r}$

If $r=0$,   
we use (3.8) in \cite{AndrewsJimenezUrrozOno}, which states that 
\begin{equation}\label{38AJUO}
4\sum_{n\geq1}(-1)^nnq^{n^2}=-2\frac{(q;q)_{\infty}}{(-q;q)_{\infty}}\sum_{n\geq1}\frac{q^n}{1-q^{2n}}+\sum_{n\geq0}\left(\frac{(q;q)_{\infty}}{(-q;q)_\infty}-\frac{(q;q)_n}{(-q;q)_n}\right).\end{equation} 
The right hand side of (\ref{38AJUO}) terminates for $q$ an odd order root of unity. Letting $\tau\mapsto\tau+\frac12$ shows that $\sum_{n\geq0}nq^{n^2}$ converges for $\tau\in\left\{\frac hk\in\Q\colon k\equiv2\pmod4\right\}$. Finally, letting $\tau\mapsto\frac N2\tau$ shows that $\Theta_{\frac32}\left(N,0;\tau\right)$ converges at the claimed rational points.

\end{proof}

\subsection{Quantum modularity of $\Theta_{\frac32}(N,r;\tau)$}\label{ProofQuantumModular}
In this subsection we use asymptotic expansions to study the quantum modularity of $\Theta^+_{\frac32}(N,r;\tau)$. For $\tau\in\H_-:=\left\{\tau\in\C\colon\text{Im}(\tau)<0\right\}$, consider the nonholomorphic Eichler integral

\begin{equation}\label{EichlerIntDefn}\Theta_{\frac12}^*(N,r;\tau):=-\frac1{2\pi\sqrt{iN}}\int_{\overline{\tau}}^{i\infty}\frac{\widetilde{\vartheta}_{\frac12}\left(N,r;z\right)}{(z-\tau)^{\frac32}}\mathrm dz. \end{equation}
We prove that $\Theta_{\frac12}^*(N,r;\tau)$ agrees (up to one elementary term) with $\Theta_{\frac32}^+(N,r;\tau)$ for $\tau=\frac hk\in\widehat{\mathcal Q}_{N,r}$ to infinite order and that $\Theta_{\frac12}^*(N,r;\tau)$ satisfies a nice transformation law for $\tau\in\mathbb{H}_-$, which demonstrates that $\Theta_{\frac32}(N,r;\tau)$ is a strong quantum modular form. Moreover, the proof clearly shows that the asymptotic expansions attached to rational numbers in the quantum set also arise from a mock modular form on the lower half-plane. 

In order to give the asymptotic expansions of these functions, we require the following periodic sequence defined for any fixed $\frac hk\in\Q$:

\begin{equation*}\gamma_{N,r}(n):=\begin{cases}e\left(\frac{hn^2}{2kN}\right)&\text{ if }n\equiv\pm r\pmod N \text{ and } r\neq0,\\2e\left(\frac{hn^2}{2kN}\right) &\text{ if }n\equiv 0\pmod N\text{ and }r=0, \\ 0&\text{ otherwise.}\end{cases}\end{equation*} 
We also need the following property of $\gamma_{N,r}(n)$:
\begin{lemma}\label{MeanValueZero}
If $N\in2\N$, $0\leq r\leq N-1$, and $\frac hk\in\widehat{\mathcal Q}_{N,r}$, then $\gamma_{N,r}(n)$ is periodic of mean value zero.
\end{lemma}
The periodicity property is clear. To prove the mean value zero property, it suffices to show that the quadratic Gauss sum $\mathcal G\left(\frac N2 h,hr,k\right)$ is zero, where for $a,b\in\Z$, $c\in\N$ we define 
\[\mathcal G(a,b,c):=\sum_{n=0}^{c-1}e\left(\frac{an^2+bn}{c}\right).\]
We can evaluate these sums using the following easily-verified facts about Gauss sums.
\begin{lemma}\label{GaussSumFacts} Let $a,b,c$ be integers. Then $\mathcal G(a,b,c)=0$ if any of the following are satisfied.
\begin{enumerate}
\item We have that $(a,c)>1$ and $(a,c)\nmid b$.
\item We have that $c\equiv0\pmod 4$ and $b$ is odd.
\item We have that $b=0$ and $c\equiv2\pmod 4$.
\end{enumerate}
\end{lemma}
We are now in position to prove Lemma \ref{MeanValueZero}.
\begin{proof}[Proof of Lemma \ref{MeanValueZero}]
For $r\not\equiv0\pmod{\frac N2}$, by definition of the quantum set we have $\frac N2\big| k$, but $hr$ is not divisible by $\frac N2$. Thus, the Gauss sum $\mathcal G\left(\frac N2h,hr,k\right)$ is zero by (1) of Lemma \ref{GaussSumFacts}.

For $r=\frac N2$, note that $\mathcal G\left(\frac N2 h,\frac N2 h,k\right)=\mathcal G(m,m,n)$, where $m:=\frac{\frac N2 h}{\left(\frac N2 ,k\right)}$, $n:=\frac{k}{\left(\frac N2 ,k\right)}$. By assumption on the quantum set, we see that $m$ is odd and $n\equiv0\pmod 4$ and so the Gauss sum vanishes by (2) of Lemma \ref{GaussSumFacts}.

Finally, if $r=0$, we reduce the Gauss sum to $\mathcal G(m,0,n)$ where $m:= \frac{\frac N2 h}{\left(\frac N2 ,k\right)}$ and $n:=\frac{k}{\left(\frac N2 ,k\right)}$. By assumption on the quantum set, we have $n\equiv 2\pmod 4$. The result follows by (3) of Lemma \ref{GaussSumFacts}.
\end{proof}

We now return to our discussion of asymptotic expansions. First we recall that for any arithmetic function $\chi\colon\Z\rightarrow\C$ and $s\in\C$, we may formally define an associated $L$-function
 \[L(s,\chi):=\sum_{n\geq1}\frac{\chi(n)}{n^s}.\]
 
 We claim that the asymptotic expansions for $\Theta_{\frac32}^+(N,r;\tau)$ and $\Theta_{\frac12}^*(N,r;\tau)$ agree. More specifically, following Lawrence and Zagier in \cite{Lawrence-Zagier} we make the following definition.
\begin{defn} Let $f(\tau)$ and $g(\tau)$ be defined for $\tau\in\H$ and $\tau\in\H^-$, respectively. We say that the asymptotic expansions of $f$ and $g$ {\bf agree} at a rational number $\frac hk$ if there exist $a_n$ such that as $t\rightarrow0^+$,

\[f\left(\frac hk+\frac{it}{2\pi}\right)\sim\sum_{n\geq0}a_nt^n,\]
\[\hspace*{0.22in}g\left(\frac hk-\frac{it}{2\pi}\right)\sim\sum_{n\geq0}a_n(-t)^n.\]
\end{defn}

We state the following. 

\begin{proposition}\label{asagree}
For any $N\in2\N$, $0\leq r\leq N-1$ and $\frac hk\in\widehat{\mathcal Q}_{N,r}$,  the asymptotic expansions for $\Theta_{\frac32}^+(N,r;\tau)$ and $\Theta_{\frac12}^*(N,r;\tau)$ agree. More specifically, we have 

\[\Theta_{\frac32}^+\left(N,r;\frac hk+\frac{it}{2\pi}\right)+\frac rNq^{\frac{r^2}{2N}}\sim\frac 1N\sum_{n\geq0}(-1)^n \frac{L(-2n-1,\gamma_{N,r})}{n!}\left(\frac{t}{2N}\right)^n,\]
\[\hspace*{0.7in}\Theta_{\frac12}^*\left(N,r;\frac hk-\frac{it}{2\pi}\right)\sim\frac 1N\sum_{n\geq0} (-1)^n\frac{L(-2n-1,\gamma_{N,r})}{n!}\left(-\frac{t}{2N}\right)^n.\]

Here $L(-2n-1,\gamma_{N,r})$ is defined by the analytic continuation of $L(s,\chi)$ to $\C$.

\end{proposition}

\noindent The key tool for proving these asymptotic expansions is the following.

\begin{lemma}\label{Zagieras}
Let $\chi\colon\Z\rightarrow\C$ be a periodic function with mean value 0. Then $L(s,\chi)$ extends holomorphically to all of $\C$ and we have as $t\rightarrow 0^+$

\[\sum_{n\geq1}n\chi(n)e^{-n^2t}\sim\sum_{n\geq0}(-1)^nL(-2n-1,\chi) \frac{t^n}{n!}.\]
If $\chi$ is even, then we also have as $t\rightarrow 0^+$
\begin{equation}\hspace*{-0.2in}\label{gammaas}\sum_{n\geq1}n\chi(n)\Gamma\left(-\frac12;2n^2t\right)e^{n^2t}\sim-2\sqrt{\pi}\sum_{n\geq0}(-1)^nL(-2n-1,\chi)\frac{(-t)^n}{n!},\end{equation}
where  $\Gamma\left(\ell ; t\right):=\int_{t}^{\infty}e^{-u}u^{\ell-1}\mathrm du.$

\end{lemma}

\begin{proof}
The first claim can be proven using standard Mellin transform techniques, for example see the proof of the proposition on page 99 of \cite{Lawrence-Zagier}. For the second claim we proceed similarly as in \cite{Lawrence-Zagier}. The idea is to compute the Mellin transform of (\ref{gammaas}) in two different ways. Making a simple change of variables and interchanging summation and integration, we see that for Re($s$)$>1$
\begin{equation}\label{MellinTrans}
\int_0^{\infty}\sum_{n\geq1}\chi(n)n\Gamma\left(-\frac12;2n^2t\right)e^{n^2t}t^{s-1}dt=L(2s-1,\chi)\int_0^{\infty}\Gamma\left(-\frac12;2t\right)e^tt^{s-1} dt.
\end{equation}
Expanding the left hand side of (\ref{MellinTrans}), we find for any $m\in\N_0$:  
\begin{equation}\label{AsympExpansionGamma}\sum_{n\geq1}\chi(n)n\Gamma\left(-\frac12;2n^2t\right)e^{n^2t}=\sum_{r=-1}^m\frac{b_{\frac r2}}{\left(\frac r2\right)!}t^{\frac r2}+O\left(t^{\frac{m+1}2}\right)\end{equation}
for certain coefficients $b_{\frac r2}$. 
The existence of this expansion can be seen using a shifted version of Proposition 3 of \cite{Zagier} (see Remark 1 following Proposition 3). We omit the details, however we remark that the proof requires the assumption that $\chi(n)$ has mean value zero together with the following behavior near $0$ and $\infty$:

\begin{equation}\label{GammaInfty}\Gamma\left(-\frac12;t\right)\sim t^{-\frac32}e^{-t}\quad\text{ as }t\rightarrow\infty,\end{equation}

\begin{equation*}\Gamma\left(-\frac12;2t\right)\sim\sqrt2t^{-\frac12}-2\sqrt{\pi}+t^{\frac12}\sum_{j\geq0}\alpha_jt^j\quad\text{ as }t\rightarrow0.\end{equation*}
In fact, we will soon show that only even $r$ occur in (\ref{AsympExpansionGamma}). Using (\ref{GammaInfty}), we see that the left hand side of (\ref{MellinTrans}) equals
\begin{equation*}\int_0^1\left(\sum_{r=-1}^m\frac{b_{\frac r2}}{\left(\frac r2\right)!}t^{\frac r2}+O\left(t^{\frac{m+1}2}\right)\right)t^{s-1}\mathrm dt+\int_1^{\infty}O\left(e^{-t}\right)t^{s-1}\mathrm dt=\sum_{r=-1}^m\frac{b_{\frac r2}}{\left(\frac r2\right)!\left(\frac r2+s\right)}+f_m(s),\end{equation*}
where $f_m$ is analytic for $\operatorname{Re}(s)>-\frac m2$.
This gives that the residue of (\ref{MellinTrans}) at a half integer $-\frac r2$ is $\frac{b_{\frac r2}}{\left(\frac r2\right)!}$. 
We can determine the coefficients $b_{\frac r2}$ by computing the residues of (\ref{MellinTrans}) using the right hand side. First note that, as $t\rightarrow0$, 
	\begin{equation}\label{OddPowerSeries}
	\left(\Gamma\left(-\frac12;2t\right)+2\sqrt{\pi}\right)e^t\sim\sqrt2t^{-\frac12}+t^{\frac12}\sum_{n\geq0}\beta_nt^n
	\end{equation} 
is an odd Laurent series in $t^{\frac12}$. To see this, we use that \[\Gamma\left(-\frac12;x\right)=\frac{2e^{-x}}{\sqrt{x}} +2\sqrt{\pi}\left(\operatorname{erf}\left(\sqrt{x}\right)-1\right),\] where $\operatorname{erf}(z):=\frac 2{\sqrt{\pi}}\int_0^ze^{-x^2}\mathrm dx$ is the usual error function and that $\operatorname{erf}(-z)=-\operatorname{erf}(z)$. 
We then compute 
\begin{equation*}\int_0^{\infty}\Gamma\left(-\frac12;2t\right)e^tt^{s-1}dt=I_0+I_1+I_2\end{equation*}
where

\[I_0:=-2\sqrt{\pi}\int_0^1e^tt^{s-1}\mathrm dt,\]

\[I_1:=\int_0^{1}\left(\Gamma\left(-\frac12;2t\right)+2\sqrt{\pi}\right)e^tt^{s-1}\mathrm dt,\]

\[I_2:=\int_1^{\infty}\Gamma\left(-\frac12;2t\right)e^tt^{s-1}\mathrm dt.\]
Using the formula

\begin{equation}\label{MellinNearZero}\int_0^1e^tt^{s-1}\mathrm dt=\sum_{j\geq0}\frac1{j!(s+j)},\end{equation}
we find that $I_0$ has simple poles precisely at $s=-j$ for $j\in\N_0$ with residue $\frac{-2\sqrt{\pi}}{j!}$. By (\ref{OddPowerSeries}), another application of (\ref{MellinNearZero}) shows that $I_1$ has no poles at negative integers. Finally, it is easy to prove that $I_2$ is entire. Thus, the residue of (\ref{MellinTrans}) at any negative integer $-j$ is $\frac{-2\sqrt{\pi}L(-2j-1,\chi)}{j!}$. Note that by the remark on page 99 of \cite{Lawrence-Zagier}, $L(-2n,\chi)=0$ for all $n\in\N$ as $\chi(n)$ is even. As $I_1$ only has simple poles, the zeros from the $L$-function cancels any pole of $I_1$, so that that the residue of (\ref{MellinTrans}) is zero at any half integer. Thus, $b_{j}$ is only non-zero for integers $j$, in which case $b_j=-2\sqrt{\pi}L(-j-1,\chi)$.

\end{proof}
We now have the necessary tools to determine the asymptotic expansions of $\Theta_{\frac32}^+(N,r;\tau)$ and $\Theta_{\frac12}^*(N,r;\tau)$. 
\begin{proof}[Proof of Proposition \ref{asagree}]
For $t>0$, we use (\ref{Theta32Sum}) to write

\[\Theta_{\frac32}^+\left(N,r;\frac hk+\frac{it}{2\pi}\right)+\frac rNq^{\frac{r^2}{2N}}=\frac1N\sum_{\substack{n>0\\ n\equiv\pm r\pmod N}}ne\left(\frac{hn^2}{2kN}\right)e^{-\frac{tn^2}{2N}}=\frac1N\sum_{n>0}n\gamma_{N,r}(n)e^{-\frac{tn^2}{2N}}.\]
This expansion, combined with 
Lemma \ref{Zagieras} and Lemma \ref{MeanValueZero}, gives the asymptotic expansion for $\Theta_{\frac32}(N,r;\tau)$ as $\tau\rightarrow\frac hk$. We next turn to $\Theta_{\frac12}^*$. By a simple change of variables and term-by-term integration, we find for $\tau\in\mathbb H_-$

\[\Theta_{\frac12}^*(N,r;\tau)=-\frac1{2N\sqrt{\pi}}\sum_{\substack{n>0\\ n\equiv\pm r\pmod N}}n\Gamma\left(-\frac12;-4\pi n^2\frac{\operatorname{Im}(\tau)}{2N}\right)q^{\frac{n^2}{2N}},\]
so as $t\rightarrow0^+$ 

\[\Theta_{\frac12}^*\left(N,r;\frac hk-\frac{it}{2\pi}\right)=-\frac1{2N\sqrt{\pi}}\sum_{n>0}n\gamma_{N,r}(n)\Gamma\left(-\frac12;\frac{2n^2t}{2N}\right)e^{\frac{n^2t}{2N}}.\]

Noting that $\gamma_{N,r}(n)$ is even, another application of Lemma \ref{Zagieras} gives the result.

\end{proof}
Finally, we describe the cocycles for $\Theta_{\frac12}^*(N,r;\tau)$. Since the proof is standard we have omitted it here; the necessary modular transformations of $\widetilde{\vartheta}_{\frac12}(N,r;\tau)$ are given in Proposition \ref{ThetaVectorTrans}.
\begin{lemma}
For any $\gamma=\left(\begin{smallmatrix}a&b\\c&d\end{smallmatrix}\right)\in\Gamma_1(2N)$, we have 
\[\Theta_{\frac12}^*(N,r;\gamma\tau)\chi_{r}(\gamma)^{-1}(c\tau+d)^{-\frac32}-\Theta_{\frac12}^*(N,r;\tau)=r_{-\frac dc}(\tau),\] 
where \[r_{x}(\tau):=\frac1{2\pi\sqrt{iN}}\int_{x}^{i\infty}\frac{\widetilde{\vartheta}_{\frac12}\left(N,r;z\right)}{(z-\tau)^{\frac32}}\mathrm dz.\]
\end{lemma}

We are now in a position to prove Theorem 1.5. 

\begin{proof}[Proof of Theorem 1.5]
By combining Lemma 6.6 and Proposition 6.4, we find that $\Theta_{\frac32}^+(N,r;\tau)$ is a strong quantum modular form. More specifically, for any $\gamma\in\Gamma_1(2N)$, we easily see that the cocycle $r_{-\frac dc}(\tau)$ extends to a real-analytic function on $\R\backslash\{\gamma^{-1}\infty\}$. As $2\Theta_{\frac32}(N,r;\tau)=\Theta_{\frac32}^+(N,r;\tau)+\Theta_{\frac32}^-(N,r;\tau)$, it suffices to prove that $\Theta_{\frac32}^-(N,r;\tau)$ is a strong quantum modular form on the same congruence subgroup and with the same multiplier. First note that any rational power of $q$ is a strong quantum modular form as its cocycle is real-analytic on $\R$.
Thus, by Proposition \ref{ThetaVectorTrans} and (\ref{Theta32Difference}), $\Theta_{\frac32}^-(N,r;\tau)$ is strongly quantum modular with the same multiplier and congruence subgroup as $\widetilde{\vartheta}_{\frac32}(N,r;\tau)$ is, which is the same as for $\Theta_{\frac32}^+(N,r;\tau)$ by Lemma 6.6.
\end{proof}


\begin{thebibliography}{99}

\bibitem{Adamovic} D. Adamovi{\'c} and O. Per{\v{s}}e, {\it Fusion rules and complete reducibility of certain modules for affine Lie algebras}, J. Algebra Appl. {\bf 13} (2014), 1, 1350062.

\bibitem{Alladi} K. Alladi, {\it Partial theta identities of Ramnaujan, Andrews, and Rogers-Fine involving the squares}, Proc. Legacy of Ramanujan, RMS-Lecture Notes Series (2013), no. 20, pp. 55--75. 

\bibitem{AndrewsConcaveConvex} G. Andrews, {\it Concave and convex compositions,} Ramanujan J. {\bf 31} (2013), no. 1-2, pp. 67--82.

\bibitem{AndrewsGarvanCrank} G. Andrews and F. Garvan, {\it Dyson's crank of a partition},  Bull. Amer. Math. Soc. (N.S.) \textbf{18} (1988), no. 2, pp. 167--171. 

\bibitem{AndrewsJimenezUrrozOno} G. Andrews, J. Jimenez-Urroz, and K. Ono, {\it $q$-series identities and values of certain $L$-functions,} Duke Math. J. \begin{bf}108\end{bf} (2001), no. 3, pp. 395--419.

\bibitem{AG} A. Atkin and F. Garvan, {\it Relations between the ranks and cranks of partitions},  Ramanujan J. {\bf7} (2003), no. 1-3, pp. 343--366. 

\bibitem{AtkinSDRank} A. Atkin and P. Swinnerton-Dyer, {\it Some properties of partitions,} Proc. London Math. Soc. (3) {\bf 4} (1954), pp. 84--106.


\bibitem{BringmannFolsomKacWaki} K. Bringmann and A. Folsom,  {\it Almost harmonic Maass forms and Kac-Wakimoto characters,} Journal f\"ur die Reine und Angewandte Mathematik (Crelle's Journal), to appear,  arXiv:1112.4726 [math.NT]. 

\bibitem{BringmannZwegers} K. Bringmann and S. Zwegers, {\it Rank-crank type PDEs and non-holomorphic Jacobi forms}, Mathematical Research Letters  {\bf 17} (2010),  pp. 589--600. 

\bibitem{BrysonPitman} J. Byrson, K. Ono, S. Pitman, R. Rhoades, {\it Unimodal sequences and quantum and mock modular forms,} Proceedings of the National Academy of Sciences {\bf 109} (2012), no. 40, pp. 16063--16067.

\bibitem{ChanDixitGarvan} S. Chan, A. Dixit, and F. Garvan, {\it Rank-crank-type PDEs and generalized Lambert series identities},  Ramanujan J. \textbf{31} (2013), no. 1-2, pp. 163--189.

\bibitem{ChernRhoades} B. Chern and R. Rhoades, {\it The Mordell integral, quantum modular forms, and mock Jacobi forms,} preprint.

\bibitem{CM}  T. Creutzig, and A. Milas, {\it False Theta Functions and the Verlinde formula}, preprint, arXiv:1309.6037 [math.QA].


\bibitem{CR1} T. Creutzig, and D. Ridout, {\it Modular Data and Verlinde Formulae for Fractional Level WZW Models I},
  Nucl. Phys. B {\bf 865} (2012), no. 1, pp. 83--114.



\bibitem{CR2} T. Creutzig, and D. Ridout, {\it Modular Data and Verlinde Formulae for Fractional Level WZW Models II},  
Nucl. Phys. B {\bf 875} (2013), no. 2, pp. 423--458. 
 

\bibitem{DGN} A. Dabholkar, D. Gaiotto, and S. Nampuri, {\it  Comments on the Spectrum of CHL Dyons},  J. High Energy Phys. (2008), no. 1, 023.

\bibitem{DMZ} A. Dabholkar, S. Murthy, and D. Zagier, {\it Quantum Black Holes, Wall Crossing, and Mock Modular Forms}, preprint, arXiv:1208.4074 [hep-th].

\bibitem{Dyson} F. Dyson, {\it Some guesses in the theory of partitions}, Eureka (Cambridge) \textbf{8} (1944), pp. 10--15.

\bibitem{Ebeling} W. Ebeling, {\it Lattices and Codes}, A course partially based on lectures by F. Hirzebruch, Advanced Lectures in Mathematics, Friedr. Vieweg \& Sohn, Braunschweig (2002). 

\bibitem{EichlerZagier} M. Eichler and D. Zagier, {\it The theory of Jacobi forms}, Progress in Mathematics {\bf 55}, Birkh\"auser Boston, Inc., Boston, MA, 1985.

\bibitem{Folsom} A. Folsom, {\it Kac-Wakimoto characters and universal mock theta functions,} Transactions of the American Mathematical Society {\bf 363} (2011), no. 1, pp. 439--455.

\bibitem{FOR} A. Folsom, K. Ono, and R. Rhoades, {\it Mock theta functions and quantum modular forms,} Forum of Mathematics, Pi \begin{bf}1\end{bf} (2013), e2, pp. 1--27.

\bibitem{Hum} J. Humphreys, {\it Introduction to Lie algebras and representation theory,} Graduate Texts in Mathematics {\bf 9}, Springer-Verlag, New York-Berlin (1978).

\bibitem{KacWakimoto} V. Kac and M. Wakimoto, {\it Integrable highest weight modules over affine superalgebras and Appell's function},  Comm. Math. Phys. {\bf 215} (2001), no. 3, pp. 631--682. 

\bibitem{KimLovejoy} B. Kim and J. Lovejoy, {\it Ramanujan-type partial theta identities and rank differences for special unimodal sequences}, preprint. 

\bibitem{KW2} V. Kac and M. Wakimoto, {\it Modular Invariant Representations of Infinite-Dimensional Lie Algebras and Superalgebras, }
Proc. Nat. Acad. Sci. USA  \begin{bf} 85 \end{bf} (1988), pp. 4956--4960.

\bibitem{KW1} V. Kac and M. Wakimoto, {\it Integrable highest weight modules over affine superalgebras and
   number theory,} Lie theory and geometry, Progr. Math. \begin{bf}123 \end{bf} (1994), pp. 415--456.
   
\bibitem{KanekoZagier} M. Kaneko and D. Zagier, {\it A generalized Jacobi theta function and quasimodular forms,} The Moduli Spaces of Curves (R. Dijkgraaf, C. Faber, G. v.d. Geer, eds.), Prog. in Math. {\bf 129}, Birkh\"auser, Boston (1995), pp. 165--172.

\bibitem{Lawrence-Zagier} R. Lawrence and D. Zagier, {\it Modular forms and quantum invariants of $3$-manifolds, } Asian J. Math. \begin{bf} 3\end{bf} (1999), no. 1, pp. 93--107.

\bibitem{LiNgoRhoades} Y. Li, H. Ngo, and R. Rhoades, {\it Renormalization and quantum modular forms, part II: Mock theta functions}, preprint, arXiv:1311.3044 [math.NT].

\bibitem{Manschot} J. Manschot,  {\it Stability and duality in N=2 supergravity,} Comm. Math. Phys. {\bf 299} (2010), no. 3, pp. 651--676

\bibitem{ManschotMoore} J. Manschot and G. Moore, {\it A modern fareytail,} Commun. Number Theory Phys. {\bf 4} (2010), no. 1, pp.103--159.

\bibitem{MeixnerRef} G. Hetyei, {\it Meixner polynomials of the second kind and quantum algebras representing $\operatorname{su}(1,1)$}, Proc. R. Soc. {\bf 466} (2010), pp. 1409--1428.

\bibitem{Ol} R. Olivetto, {\it On the Fourier coefficients of meromorphic Jacobi forms}, International Journal of Number Theory, to appear, arXiv:1210.7926 [math.NT].

\bibitem{Rademacher} H. Rademacher, {\it Topics in analytic number theory,} Die Grundlehren der math. Wiss., Band 169, Springer-Verlag, Berlin (1973).

\bibitem{Sen} A. Sen, {\it Negative discriminant states in N=4 supersymmetric string theories},  J. High Energy Phys. {\bf 073} (2011), no. 10, pp. 1--29.

\bibitem{Shimura} G. Shimura, {\it On modular forms of half integral weight}, Ann. of Math. (2) \begin{bf} 97\end{bf} (1973), pp. 440--481.

\bibitem{Stanley} R. Stanley, {\it Enumerative Combinatorics,} Vol. 1. Cambridge: Cambridge University Press; 1997. Cambridge Studies in Advanced Mathematics {\bf 49}.

\bibitem{Warnaar} O. Warnaar, {\it Partial theta functions. I. Beyond the lost notebook,} Proc. London Math. Soc. {\bf87} (2003), pp. 363--395.

\bibitem{Wall} H. Wall, {\it Analytic theory of continued fractions}, Chelsea Publishing Co., Bronx, NY, 1967.

\bibitem{W} H. Weyl, {\it Theorie der Darstellung kontinuierlicher halb-einfacher Gruppen durch lineare Transformationen. I,} Mathematische Zeitschrift \begin{bf} 23 \end{bf} (1925), pp. 271--309.

\bibitem{Wright} E. Wright, {\it Stacks,} Quarterly Journal of Mathematics Oxford Series {\bf 19} (2) 1968, pp. 313--20.

\bibitem{Zagier} D. Zagier, {\it The Mellin transform and other related analytic techniques,} Appendix to E. Zeidler, Quantum Field Theory I: Basics in Mathematics and Physics. A Bridge Between Mathematicians and Physicists
Springer-Verlag, Berlin-Heidelberg-New York (2006), pp. 305--323.

\bibitem{ZagierQuantum} D. Zagier, {\it Quantum modular forms}, Clay Math. Proc. \textbf{11}, Amer. Math. Soc., Providence, RI, 2010. 

\bibitem{ZagierVassiliev} D. Zagier, {\it Vassiliev invariants and a strange identity related to the Dedekind eta-function}, Topology \textbf{40} (2001), no. 5, pp. 945--960.

\bibitem{ZwegersThesis} S. Zwegers, {\it Mock Theta Functions} (2002), Utrecht PhD thesis.

\bibitem{ZwegersRankCrankPDE} S. Zwegers, {\it Rank-crank type PDE's for higher level Appell functions,} Acta Arith. \begin{bf} 144\end{bf} (2010), no. 3, pp. 263--273.

\end{thebibliography}
\end{document}